\documentclass[12pt]{amsart}

\usepackage[text={400pt,640pt},centering]{geometry}

\usepackage{color}
\usepackage{esint,amssymb}
\usepackage{graphicx}
\usepackage{MnSymbol}
\usepackage{mathtools}
\usepackage[colorlinks=true, pdfstartview=FitV, linkcolor=blue, citecolor=blue, urlcolor=blue,pagebackref=false]{hyperref}
\usepackage{microtype}

\definecolor{darkgreen}{rgb}{0,0.5,0}
\definecolor{darkblue}{rgb}{0.1,0.1,0.9}
\definecolor{darkred}{rgb}{0.9,0.1,0.1}

\newtheorem{theorem}{Theorem}
\newtheorem{proposition}[theorem]{Proposition}
\newtheorem{lemma}[theorem]{Lemma}

\theoremstyle{definition}
\newtheorem{remark}[theorem]{Remark}

\newcommand{\pref}[1]{Proposition~\ref{p.#1}}

\newcommand{\cref}[1]{Corollary~\ref{c.#1}}

\newcommand{\sref}[1]{Section~\ref{s.#1}}
\newcommand{\ssref}[1]{Subsection~\ref{ss.#1}}

\numberwithin{equation}{section}
\numberwithin{theorem}{section}

\newcommand{\Z}{\mathbb{Z}}
\newcommand{\N}{\mathbb{N}}
\newcommand{\R}{\mathbb{R}}

\newcommand{\Zd}{\mathbb{Z}^d}
\newcommand{\Rd}{\mathbb{R}^d}

\newcommand{\ep}{\varepsilon}

\renewcommand{\a}{\mathbf{a}}

\renewcommand{\fint}{\strokedint}

\DeclareMathOperator*{\osc}{osc}

\renewcommand{\bar}{\overline}
\renewcommand{\tilde}{\widetilde}

\DeclareMathOperator{\data}{data}
\newcommand{\diff}{\Delta}

\begin{document}

\title[Bounded correctors in almost periodic homogenization]{Bounded correctors in almost periodic homogenization}

\begin{abstract}
We show that certain linear elliptic equations (and systems) in divergence form with almost periodic coefficients have bounded, almost periodic correctors. This is proved under a new condition we introduce which quantifies the almost periodic assumption and includes (but is not restricted to) the class of smooth, quasiperiodic coefficient fields which satisfy a Diophantine-type condition previously considered by Kozlov~\cite{K}. The proof is based on a quantitative ergodic theorem for almost periodic functions combined with the new regularity theory recently introduced by the first author and Shen~\cite{ASh} for equations with almost periodic coefficients. This yields control on spatial averages of the gradient of the corrector, which is converted into estimates on the size of the corrector itself via a multiscale Poincar\'e-type inequality.
\end{abstract}

\author[S. Armstrong]{Scott Armstrong}
\address[S. Armstrong]{Universit\'e Paris-Dauphine, PSL Research University, CNRS, UMR [7534], CEREMADE, 75016 Paris, France}
\email{armstrong@ceremade.dauphine.fr}

\author[A. Gloria]{Antoine Gloria}
\address[A. Gloria]{D\'epartement de math\'ematique, Universit\'e Libre de Bruxelles, Belgium and project-team MEPHYSTO, Inria Lille-Nord Europe, France}
\email{agloria@ulb.ac.be}

\author[T. Kuusi]{Tuomo Kuusi}
\address[T. Kuusi]{Department of Mathematics and Systems Analysis, Aalto University, Finland}
 \email{tuomo.kuusi@aalto.fi}

\keywords{}
\subjclass[2010]{}
\date{\today}

\maketitle


\section{Introduction}
\label{s.introduction}
\subsection{Motivation and informal summary of results}
We consider uniformly elliptic equations with almost periodic coefficients, taking the form
\begin{equation} \label{e.pde}
-\nabla \cdot \left( \a\left(x\right) \nabla u \right) = 0 \qquad \mbox{in} \ U\subseteq \Rd.
\end{equation}
The coefficient field $\a:\Rd \to \R^{d\times d}$ is assumed to satisfy, for every $x,\xi\in\Rd$,
\begin{equation} \label{e.ue}
\xi\cdot \a(x) \xi \geq \left| \xi \right|^2 \quad \mbox{and} \quad \left| \a(x)\right| \leq \Lambda,
\end{equation}
where $\Lambda \geq 1$ is a given parameter. It is also assumed to be \emph{uniformly almost periodic}, that is,
\begin{equation} \label{e.ap}
\lim_{R\to \infty} \rho_1(\a,R) = 0,
\end{equation}
where here and throughout the paper we denote, for each $R>0$ and bounded, continuous function $f:\Rd \to \R$,
\begin{equation} \label{e.rho}
\rho_1(f,R):= \sup_{y\in\Rd} \inf_{z\in B_R} \sup_{x\in\Rd} \left| f(x+y) - f(x+z) \right|.
\end{equation}
Uniform almost periodicity in the sense above is equivalent to limits (in $L^\infty(\Rd)$) of sequences of trigonometric polynomials (see~\cite{B} or~\cite{JKO}). Notice that an $L$-periodic function $f$ satisfies $\rho_1(f,R)=0$ for every $R \geq L\sqrt{d}$. 

\smallskip

The study of~\eqref{e.pde} when $U$ is a very large domain falls into the realm of homogenization. 
 This paper is focused on the question of the boundedness of correctors, a central topic in homogenization and closely related to the issue of obtaining rates of convergence for homogenization. Our main result asserts the existence, for each unit vector $e\in\partial B_1$, of a bounded (and hence almost periodic) solution $\phi$ of the equation
\begin{equation*} \label{}
-\nabla \cdot \left( \a(x) \left( e + \nabla \phi (x) \right) \right) = 0 \quad \mbox{in} \ \Rd
\end{equation*}
under certain quantitative ergodic assumptions on $\a(\cdot)$ stated below.

\smallskip

This paper contains the first improvement of the results of Kozlov~\cite{K}, now nearly four decades old yet still unsurpassed, who proved the existence of bounded correctors for \emph{quasiperiodic} coefficients satisfying certain smoothness and non-resonance (Diophantine) conditions. Kozlov's work was still the only positive result on the existence of bounded correctors outside of the periodic setting (where the result is of course trivial). He lifted the corrector equation to a sub-elliptic problem on the higher dimensional torus which he solved thanks to a higher-order Poincar\'e inequality implied by the Diophantine condition. His ideas have been used in similar contexts in homogenization, see for example~\cite{GVM1}.

\smallskip

Our method is very different from that of Kozlov and applies to more general coefficient fields. 
In particular, in addition to recovering his result, we are able to identify the first class of almost periodic coefficients which are not quasiperiodic but for which bounded correctors exist. 
Our approach is inspired by recent advances on corresponding problems in stochastic homogenization, which we quickly describe now.

\smallskip

In a series of papers, the second author and Otto~\cite{GO11,GO14} and the second author, Neukamm and Otto~\cite{GNO1} proved new estimates on the (approximate) correctors under some strong mixing assumptions on the random coefficient field $\a(\cdot)$. Their strategy can be roughly summarized in three steps:

\begin{enumerate}
\item[(I)]  Introduction of a differential calculus with respect to the coefficient field $\a$, which we informally denote by $\partial_\a$, which possesses two important properties: (i) $\partial_\a$ quantifies ergodicity and (ii) the differential calculus is compatible with the PDE in the sense that if $u$ is a solution of the PDE with coefficients $\a$, then $\partial_\a u$ is itself solution of a similar PDE.

\smallskip

\item[(II)]  Prove higher-regularity results for the approximate correctors $\phi_\ep$ (which are defined by~\eqref{e.appcorrectors} below). Roughly what is needed is an $L^\infty$ bound on $\nabla \phi_\ep$. This makes it possible to transfer ergodic information from the coefficient field to the gradients of the solutions. 

\smallskip

\item[(III)] From (I) and (II), deduce bounds on spatial averages of the gradients of the approximate correctors and then ``integrate" these to obtain bounds on the oscillation of the approximate correctors themselves. 

\end{enumerate}

In~\cite{GO11,GNO1,GO14}, the differential calculus is based on Glauber derivatives and the ergodicity properties are quantified by functional inequalities on the probability space such as the spectral gap or logarithmic Sobolev inequalities. Higher regularity estimates for the approximate correctors are proved using this differential calculus and exploiting the specific structure of the corrector equation, while the final step is obtained by combining~(I) and~(II) and using estimates on Green's functions.

\smallskip

The fundamental difference between the random setting treated in the papers mentioned above and the almost periodic setting is the origin of cancellations: in the random setting, cancellations occur due to the \emph{decorrelation} properties of $\a$ at large distances, whereas in the almost periodic setting, they occur due to the \emph{correlation} properties of $\a$ at large distances. Nevertheless, our strategy of proof in the almost periodic setting follows these three steps. 

\smallskip

In the case of almost periodic coefficients, the role of~$\partial_\a$ is played by a ``difference operator" which we denote by~$\diff_{yz}$ and which measures the sensitivity of functions~$X$ of the almost periodic coefficient field~$\a(\cdot)$ with respect to translations of~$\a(\cdot)$ and therefore monitors the long-range correlations of~$X$. The difference operator~$\Delta_{yz}$ and the associated (standard) quantification~$\rho_1$ of almost periodicity (cf.~\eqref{e.rho}) is too coarse of a measure of ergodicity for what we need. To refine the quantification of ergodicity in the almost setting, we iterate the difference operator~$\Delta_{yz}$ to any order, see~\eqref{e.iterate-diff} below (in the spirit of the Hoeffding decomposition in sensitivity analysis). This gives rise to a new modulus (which we call~$\rho_*(\a,\cdot)$, based on a family $\rho_k(\a,\cdot)$, see~\eqref{e.defrhostar} and~\eqref{e.defrhok}) which  quantifies the almost periodicity of~$\a(\cdot)$ and controls the rate of convergence in the ergodic theorem (see~\pref{ergodicthm} below). The difference operator~$\Delta_{yz}$ and its iterates define a suitable calculus for PDEs with almost periodic coefficients. In particular, the function~$\Delta_{yz} \phi_\ep$ is a solution of an equation with almost periodic coefficients, see~\eqref{e.Delta-phi}.


\smallskip

As far as step (II) is concerned, a regularity theory has already been proved by Shen and the first author in~\cite{ASh} (see~\pref{lipschitz} below). This is an extension of the one developed in the periodic case by Avellaneda and Lin~\cite{AL1,AL2}. The method of proof in~\cite{ASh} was however different and also based on recent advances in stochastic homogenization, in particular the new quantitative arguments for Lipschitz regularity originating in~\cite{AS} and further developed in~\cite{AM,GNO2}. See also the earlier works~\cite{MO,S} which developed a H\"older regularity theory. 

\smallskip

The combination of the Lipschitz regularity theory and the difference calculus of step (I) allows us to transfer the quantitative almost periodic assumptions from $\a$ to $\nabla \phi_\ep$ in the form of \eqref{e.modcontrol}. An application of the quantitative ergodic theorem (\pref{ergodicthm}) yields control of spatial averages of~$\nabla \phi_\ep$ and this is integrated using a multiscale Poincar\'e inequality (\pref{heatfuckingflow}) to give bounds on the oscillation of $\phi_\ep$. If the rate in the ergodic theorem is sufficient, then we deduce that $\osc_{\Rd} \phi_\ep$ is bounded uniformly in $\ep\in(0,1]$ and thus the existence of a bounded corrector by sending $\ep \to 0$. 

\smallskip

This string of arguments therefore gives a sufficient condition for the existence of a bounded corrector in terms of the decay of the modulus $\rho_*(\a,\cdot)$ of the coefficients. To apply this theory to Kozlov's class of quasiperiodic coefficients, for example, it is necessary only to check that the Diophantine condition (and sufficient smoothness) implies that $\rho_*(\a,\cdot)$ decays sufficiently fast. Indeed it does, as we show in Section~\ref{s.examples}, where we actually obtain the results for a strictly larger class of almost periodic coefficient fields.

\smallskip

We remark that, while we use scalar notation throughout the paper, the arguments work essentially verbatim in the case of systems.

\subsection{Quantitative almost periodicity}
\label{ss.almostperiodic}
In this subsection we introduce the quantitative almost periodic conditions under which we prove our main results. We begin with some notation. 
Given $f:\Rd \to \R^k$ and $y,z\in\Rd$, we define
\begin{equation*}
T_zf(x):= f(x+z)
\end{equation*}
and the difference operator
\begin{equation}
\label{e.diff-op}
\diff_{yz}f(x):= \frac12 \left( T_yf(x) - T_zf(x)\right) = \frac12 \left(f(x+y) - f(x+z)\right).
\end{equation}
Then the quantity $\rho_1(f,R)$ defined in the previous section can be written as
\begin{equation*}
\rho_1(f,R)= 2 \sup_{y\in\Rd} \inf_{z\in B_R} \left\| \diff_{yz} f \right\|_{L^\infty(\Rd)}.
\end{equation*}
This quantity is a natural way to quantify the almost periodic assumption (at least for uniformly almost periodic functions) and it has been used already for example in~\cite{ACS,S} to obtain quantitative estimates in almost periodic homogenization.

\smallskip

Our purposes require a more refined quantitative measurement of almost periodicity. The reason is that $\rho_1(\a,R)$ is bounded below by $cR^{-1}$, even in the best (non-periodic) situations such as the quasiperiodic case with a Diophantine condition (see Section~\ref{s.examples} below). Unfortunately, the assumption $\rho_1(\a,R) \lesssim R^{-1}$ does not imply the existence of bounded correctors. This motivates us to consider higher-order versions of $\rho_1$, which are defined in the  following way, where we use $\| \cdot\| = \| \cdot \|_{L^\infty(\Rd)}$ for brevity:
\begin{align*}
 \rho_2(f,R) 
& := \sup_{y_1\in\Rd} \inf_{z_1\in B_R} \sup_{y_2\in\Rd} \inf_{z_2\in B_R} \max \left\{ \left\|\diff_{y_2z_2} \diff_{y_1z_1} f \right\| ,\, \left\|\diff_{y_1z_1}f \right\| \left\| \diff_{y_2 z_2} f \right\|   \right\},
\\
 \rho_3(f,R)
& :=  \sup_{y_1\in\Rd} \inf_{z_1\in B_R} \sup_{y_2\in\Rd} \inf_{z_2\in B_R}\sup_{y_3\in\Rd} \inf_{z_3\in B_R} \\
& \qquad \max \big\{ 
\left\|\diff_{y_3z_3} \diff_{y_2z_2} \diff_{y_1z_1} f \right\|,\, \left\|\diff_{y_2z_2} \diff_{y_1z_1}f \right\|\left\| \diff_{y_3 z_3} f \right\|, \\
& \qquad \qquad \qquad \left\|\diff_{y_3z_3} \diff_{y_1z_1}f \right\| \left\|  \diff_{y_2 z_2} f \right\|, \, \left\|\diff_{y_3z_3} \diff_{y_2z_2}f \right\| \left\| \diff_{y_1 z_1} f \right\|  ,  \\
& \qquad \qquad \qquad  \qquad \left\|\diff_{y_1z_1}f \right\| \left\| \diff_{y_2 z_2} f \right\| \left\| \diff_{y_3 z_3} f \right\|    \big\}.
\end{align*}
In order to define $\rho_k$ for higher $k\in\N$, it is (unfortunately) necessary to develop some notation to keep track of the combinatorics. 

\smallskip

Let $\mathcal T_k = \left( (y_1,z_1),\ldots,(y_k, z_k) \right) \in (\Rd\times\Rd)^k$ be a $k$-tuple formed by couples $(y_j ,z_j) \in \R^d \times \R^d$. For a function $f:\R^d \to \R^{m \times n}$, $m,n \in \N$, we define a difference operator $\Delta_{{\mathcal T}_k}$ acting on $f$ by
\begin{equation}
\label{e.iterate-diff}
\diff_{\mathcal T_k} f(x) = \diff_{y_k z_k} \cdots \diff_{y_1 z_1} f(x)\,.
\end{equation}
In what follows, we need to control composite quantities of $L^\infty(\R^d;\R^{m \times n})$-norms of differences, and for this we need a proper way to describe partitions of $\{1,\ldots,k\}$. 
Let $\mathcal P_{j,k}$, $j \in \{0,1,\ldots,k\}$, stand for a set of increasing ordered subsets of $\{1,\ldots,k\}$ with $j$ members. In other words, for $j>0$ we define
\begin{equation*} \label{}
\mathcal P_{j,k}:= \left\{ \zeta \in \left\{ 1,\ldots,k\right\}^j \,:\,    \zeta_i < \zeta_{i+1}  \  \forall  i \in \{ 1,\ldots,j-1\} \right\}
\end{equation*}
and, for $j=0$, we set $\mathcal P_{0,k} = \emptyset$. By abuse of notation, we also think of $\zeta \in \mathcal{P}_{j,k}$ as being ordered subsets of $\{ 1,\ldots,k\}$. Then, for $\zeta \in \mathcal P_{j,k}$, we denote by $\zeta^c$ the unique member of $\mathcal P_{k-j,k}$ such that $\{1,\ldots,k\} = \zeta \cup \zeta^c$. By $|\zeta|$ we denote the number of elements in $\zeta \in P_{j,k}$, i.e., $|\zeta| = j$. For $\mathcal T_k$ as above and for $\zeta \in \mathcal P_{j,k}$ we denote by the $j$-tuple $\zeta(\mathcal T_k)$  the set  $\left( ( y_{\zeta_1}, z_{\zeta_1}), \ldots, (y_{\zeta_{j}}, z_{\zeta_{j} }) \right)$ for $j>0$, and if $\zeta \in \mathcal P_{0,k}$, we set $\zeta(\mathcal T_k) = \emptyset$  and  $\diff_{\zeta(\mathcal T_k)} f = 1$.  Furthermore, we let $\mathcal P_k$ stand for the family of subsets $(\zeta^1,\ldots,\zeta^k) \in \mathcal P_{j_1,k} \times \cdots \times \mathcal P_{j_k,k}$ with $\sum_{i=1}^k j_i = k$.

 \smallskip

We are now ready to define two central concepts of the paper. For a given $f \in C(\R^d;\R^{m \times n})$, $m,n \in \N$, and $\mathcal T_k = \{(y_1,z_1),\ldots,(y_k, z_k)\}$ we define 
\begin{equation} \label{e.G_k}
G_k(f,\mathcal T_k) := \max_{(\zeta^1,\ldots,\zeta^k) \in \mathcal P_k} \left\{ \prod_{j=1}^k \left\| \diff_{\zeta^j(\mathcal T_k)} f  \right\|_{L^\infty(\R^d;\R^{m \times n})} \right\}\,,
\end{equation}
that is, the maximum is taken over all (increasing, ordered) partitions of $\mathcal P_k$. Finally, we define $\rho_k$, for each $k \in \N$ and $R \geq 1$, by
\begin{equation} \label{e.defrhok}
 \rho_k(f,R) : = \sup_{y_1\in\Rd} \inf_{z_1\in B_R} \cdots \sup_{y_k\in\Rd} \inf_{z_k\in B_R} G_k\left(f , \left( (y_1, z_1),\ldots, (y_k,z_k)†\right) \right).
\end{equation}

\smallskip

The main justification that $\rho_k(\a,R)$ is a good quantification of the almost periodicity assumption comes in Section~\ref{s.ergodicthm} where we use it to obtain a quantitative ergodic theorem for almost periodic functions, a key ingredient in our arguments. On the other hand, the $\rho_k$'s are useful in our context because, as will show in \ssref{phiep}, we can essentially estimate $\rho_k(\nabla \phi,R)$ for a corrector $\phi$ in terms of $\rho_k(\a,R)$. To show that $\phi$ itself is bounded, it suffices to have $\rho_k(\a,R) \lesssim R^{-1-\delta}$ for some $\delta > 0$. 

\smallskip

The modulus $\rho_k(\a,R)$ tends to decay faster in $R$ as $k$ becomes larger, but the prefactor constants become larger in $k$. Indeed, in the case of (sufficiently smooth) quasiperiodic coefficients with frequencies satisfying a Diophantine condition, we have an (essentially sharp) bound like $\rho_k(\a,R) \lesssim C^kR^{-ck}$ (see Section~\ref{s.examples}). In the general almost periodic case, we will not have power-like decay in~$R$ of any $\rho_k(a,R)$. Therefore, to obtain the necessary estimates, it is necessary to choose $k$ depending on $R$. This motivates us to define
\begin{equation}
\label{e.defrhostar}
\rho_*(\a,R):= \inf_{k \in\N \cap [1,R] } C^k k! \rho_k\left(\a, k^{-1} R\right).
\end{equation}
The constant $C$ in~\eqref{e.defrhostar} depends on the parameters in the assumptions (besides the assumption on $\rho_*$) and can be computed by an inspection of the proofs.
The main quantitative ergodicity assumption that we make on the coefficients is therefore that there exists an exponent $\delta > 0$ and a constant $K\geq 1$ such that, for every $R\geq 1$, 
\begin{equation}
\label{e.ergodicassump}
\rho_*(\a,R) \leq K R^{-1-\delta}.
\end{equation}
In Section~\ref{s.examples}, we give a general class of almost periodic coefficients satisfying~\eqref{e.ergodicassump} which is strictly larger than the set of quasiperiodic coefficients satisfying a Diophantine condition. In fact, under Kozlov's condition, it is easy to show that the modulus $\rho_*(\a,R)$ decays faster than any finite power of~$R$.

\subsection{Main results}

Before giving the statement of the main result, we review the assumptions on the coefficient field. We require that~$\a:\Rd\to\R^{d\times d}$ satisfy~\eqref{e.ue} and~\eqref{e.ergodicassump} as well as, for some $K\geq1$, $\gamma\in (0,1]$ and $\kappa > \frac 52$ and every $x,y\in\Rd$ and $R\geq 1$,
\begin{equation}
\label{e.aholder}
\left| \a(x) - \a(y) \right| \leq K\left|x-y\right|^\gamma. 
\end{equation}
and
\begin{equation}
\label{e.technicalassump}
\rho_1(\a,R) \leq K \left( \log R \right)^{-\kappa}.
\end{equation}
The latter two assumptions are purely technical. 
We make them in order to directly apply the regularity theory developed in~\cite{ASh}, since the assumption is made in that paper, rather than reprove the needed estimates using only~\eqref{e.ergodicassump}. Indeed, while we do not give the details here, we expect that, by the methods of~\cite{ASh}, a decay assumption like $\rho_*(R) \lesssim R^{-\theta}$ for some $\theta>0$ (or even a weaker Dini-type condition) would suffice to yield Lipschitz and $W^{1,p}$ estimates of the sort proved in~\cite{ASh} (and summarized below in Proposition~\ref{p.lipschitz}). Even the H\"older condition~\eqref{e.aholder} can probably be removed, as the regularity estimates can survive in an appropriate form without it (see~\cite{AS}). 

\smallskip

For brevity, here and throughout the paper we collect all the constants in these assumptions by denoting
\begin{equation*}
\data:= \left(d,\Lambda,K,\delta,\kappa,\gamma\right). 
\end{equation*}

\smallskip

We next present the main result. 

\begin{theorem}
\label{t.correctors}
Suppose that the coefficient field $\a$ satisfies~\eqref{e.ue},~\eqref{e.ergodicassump},~\eqref{e.aholder} and~\eqref{e.technicalassump}. Then, for each unit vector $e\in\partial B_1$, there exists an almost periodic function $\phi \in W^{1,\infty}(\Rd)$ satisfying the equation
\begin{equation*}
-\nabla \cdot \left( \a(x) \left(e + \nabla \phi \right) \right) = 0\quad \mbox{in} \ \Rd. 
\end{equation*}
Moreover, there exists a constant $C(\data) < \infty$ such that 
\begin{equation}
\label{e.boundedcorrector}
\sup_{\Rd} \left| \phi  \right| \leq C.
\end{equation}
\end{theorem}

\smallskip

We believe that Theorem~\ref{t.correctors} cannot be improved in terms of $\rho_*(\a,R)$, that is, an assumption $\rho_*(\a,R) \lesssim R^{-1}$ does not imply the existence of a bounded corrector. We do not however justify this belief by constructing an example here.
On the other hand the required decay of $\rho_*$ is a weak enough condition that is satisfied by a class of coefficients that includes some almost periodic functions that are not quasi-periodic, see discussion Section~\ref{s.examples}.
In addition, the arguments in this paper do yield bounds on the sublinear growth of $\phi$ under the weaker condition (than \eqref{e.ergodicassump}) that, for some $\theta \in (0,1]$, 
\begin{equation*} \label{}
\rho_*(\a,R) \leq K R^{-\theta}.
\end{equation*}
In this case we can show, by inspecting the arguments in \ssref{phiep}, that, for every $\alpha>0$, there exists $C(\alpha,\data) \geq 1$ such that the (possibly unbounded) corrector $\phi$ satisfies, for every $R\geq 1$, 
\begin{equation} \label{e.sublinearcorrector}
\sup_{x\in \Rd} \sup_{y \in B_R(x)} \left| \phi(x) - \phi(y) \right| \leq C R^{1-\theta+\alpha}.
\end{equation}

\smallskip

In terms of homogenization, an estimate like~\eqref{e.boundedcorrector} implies an $O(\ep)$ rate of convergence in homogenization. For example, if we denote the effective coefficients by~$\overline\a$ and $U\subseteq \Rd$ is a smooth domain, $g:U\to \R$ is sufficiently smooth and $u^\ep,u\in g + H^1_0(U)$ are solutions of
\begin{equation*}
 -\nabla \cdot \left( \a\left( \frac x\ep \right) \nabla  u^\ep \right) = 0  \qquad \mbox{and} \qquad -\nabla  \cdot \left( \overline \a \nabla u \right) = 0  \qquad \mbox{in}  \ U,
\end{equation*}
then we have the estimate
\begin{equation*}
\left\| u^\ep - u \right\|_{L^\infty(U)} \leq O(\ep)
\end{equation*}
where the implicit prefactor constant depends only on $(\data,U,g)$. Moreover, one can use the corrector to obtain estimates on the two-scale expansion (away from boundary layers): for any $p<\infty$ and $V\Subset U$,
\begin{equation*}
\left\| u^\ep - u - \ep \nabla u\cdot \Phi\left( \tfrac\cdot\ep \right) \right\|_{W^{1,p}(V)} \leq O(\ep),
\end{equation*}
where $\Phi$ is the vector $(\phi_{e_1},\ldots,\phi_{e_d})$ of correctors in the unit directions and the prefactor constant depends additionally  on $p$ and $V$. Similarly, an estimate like~\eqref{e.sublinearcorrector} for the sublinear growth of the corrector implies the same estimates summarized above but with $O(\ep^{\theta-})$ in place of $O(\ep)$. The proofs of these facts are identical to those in the periodic case and thus very classical (see, for example~\cite{BLP} or~\cite{AL1}), so we do not give the details here. 

\smallskip

Finally, we mention that the existence and boundedness of higher-order correctors can also be established using the arguments in this paper, justifying more terms in the \emph{a priori} two-scale expansion in homogenization. The statement is roughly that  for each $k\in\N$, a $k$th order corrector exists and is bounded under the assumption that $\rho_*(\a,R) \lesssim R^{-k-\delta}$ for some $\delta >0$.

\subsection{Outline of the paper}
In the next section, we present a quantitative ergodic theorem for almost periodic functions in terms of the $\rho_k$ defined in Section~\ref{ss.almostperiodic}. This is separate from the rest of the paper and is of independent interest. In Section~\ref{s.proofs} we give the proof of Theorem~\ref{t.correctors}. In Section~\ref{s.examples} we give examples of almost periodic coefficient fields satisfying our assumptions and prove the sharpness of our results.

\section{Quantitative ergodic theorem for almost periodic functions}
\label{s.ergodicthm}

The next proposition is a quantitative ergodic theorem for uniformly almost periodic functions. It roughly gives a convergence rate for the spatial averages of an almost periodic function to its mean where, motivated by \pref{heatfuckingflow} in \sref{poinca} below, we express spatial averages in terms of heat flow. While we expect that  results of a similar flavor are known, perhaps by Fourier methods, we could not find a similar statement in the literature. Here and throughout the paper, $\Phi$ denotes the standard heat kernel 
\begin{equation}
\label{e.heatkernel}
\Phi(x,t):=(4\pi t)^{-\frac d2} \exp\left( - \frac{|x|^2}{4t}\right).
\end{equation}

\smallskip

In the following statement, we use a slightly different way of measuring almost periodicity from the $\rho_k$'s defined in~\eqref{e.defrhok}. Compared to $\rho_k$, in view of future applications, we need to relax the $L^\infty$ norm to a hybrid between $L^1$ and $L^\infty$ norms. On the other hand, we only need to use one term in the maximum in the definition of $\rho_k$. We denote the unit cube by $Q:= [0,1]^d$ and define, for $f\in  L^1_{\mathrm{loc}} (\Rd)$, $R\geq 1$ and $k\in\N$:

\begin{equation}
\omega(f,\mathcal T_k) := 
\sup_{z' \in \R^d} \int_{B_1(z')}  \left| \diff_{\mathcal T_k} f(x) \right|\,dx
\end{equation}
and
\begin{equation}
\omega_k(f,R) := \sup_{y_1\in\Rd} \inf_{z_1\in B_R}\cdots \sup_{y_k\in\Rd} \inf_{z_k\in B_R} \omega(f,\left((y_1,z_1),\ldots,(y_k,z_k)) \right)) .
\end{equation}
We shall prove in \ssref{phiep} that $\rho_k$ essentially controls $\omega_k$ in our applications to homogenization, cf. \eqref{e.modcontrol}.
If $f:\Rd \to \R^k$ is a uniformly almost periodic function in the sense that $f$ is continuous, bounded and $\rho_1(f,R) \to 0$ as $R\to \infty$, then $f$ has a \emph{mean value} which we denote by $\left\langle f \right\rangle$. It is characterized for example by the fact that 
\begin{equation*}
\left\langle f \right\rangle = \lim_{R\to \infty} \fint_{B_R} f(x)\,dx.
\end{equation*}
We refer to~\cite{B} for more on almost periodic functions.

\begin{proposition}
\label{p.ergodicthm}
Assume $f \in L^1_{\mathrm{loc}}(\Rd)$ is uniformly almost periodic in the sense that 
\begin{equation}
\label{e.falmostper}
\rho_1(f,R) \to 0 \quad \mbox{as} \ R \to \infty. 
\end{equation}
Let $u$ be the solution of the Cauchy problem
\begin{equation*}
\left\{ \begin{aligned}
& \partial_t u - \Delta u = 0 & \mbox{in} & \ \Rd \times (0,\infty), \\
& u(\cdot,0) = f & \mbox{in} & \ \Rd. 
\end{aligned} \right.
\end{equation*}
Then there exist constants $C\in [1,\infty)$ and $c\in (0,1]$, depending only on~$d$, such that, for every $k\in\N$ and $t \in [k,\infty)$,
\begin{equation}
\label{e.ergthm1}
\osc_{\Rd} u(\cdot,t) \leq C^k \inf_{R\geq 1} \left( \omega_k(f,R) + \exp\left( - \frac{ct}{kR^2} \right)\sup_{z'\in \R^d} \| f \|_{L^1(B_1(z'))}\right)
\end{equation}
and
\begin{equation}
\label{e.ergthm2}
\sup_{\Rd} \left| \nabla u(\cdot,t) \right| \leq C^k \inf_{R\geq 1} \left( t^{-\frac12} \omega_k(f,R) +  \exp\left( - \frac{ct}{kR^2} \right)\sup_{z'\in \R^d} \| f \|_{L^1(B_1(z'))} \right).
\end{equation}
\end{proposition}
\begin{proof}
We first prove Proposition~\ref{p.ergodicthm} in the case $k=1$ before obtaining the statement for general $k\in\N$ by an induction argument. For convenience, we denote
\begin{equation*}
\| f \| := \sup_{z' \in \Zd} \| f \|_{ L^1(z' + Q)} \leq C \sup_{z'\in \R^d} \| f \|_{L^1(B_1(z'))}.
\end{equation*}
Note that for all $t\ge 0$, $u(\cdot,t)$ is uniformly almost periodic (this can be seen using formula \eqref{e.rep-form} below).

\smallskip

\emph{Step 1.} The proof in the case $k=1$. We use the representation formula
\begin{equation}\label{e.rep-form}
u(x,t) = \int_{\Rd} f(y) \Phi(x-y,t) \, dt,
\end{equation}
where $\Phi$ is the heat kernel defined in~\eqref{e.heatkernel}. Fix $t\geq 1$, $y_1 \in \Rd$ and $R\geq 1$. Select $z_1 \in B_R$ such that,
with $\mathcal T_1 = \left( (y_1,z_1) \right)$,
\begin{equation*}
\sup_{z' \in\R^d} \int_{B_1(z')} \left| \diff_{\mathcal T_1} f(x) \right|\,dx \leq \omega_1(f,R).
\end{equation*}
Observe that, for every $m \in\N$, 
\begin{align*}
\left| \diff_{\mathcal T_1} \nabla^m u(0,t) \right| 
& = \left| \int_{\Rd} \diff_{\mathcal T_1} f(x)  \nabla^m \Phi(x,t)  \,dx  \right| \\
& = \left| \sum_{z'\in\Zd} \int_{z'+Q} \diff_{\mathcal T_1} f(x)\nabla^m \Phi(x,t)  \,dx  \right| \\
& \leq \sum_{z'\in\Zd}  \left\| \nabla^m\Phi(x,t) \right\|_{L^\infty(z'+Q)} \left( \int_{z'+Q} \left| \diff_{\mathcal T_1} f(x) \right|\,dx \right) \\
& \leq C \omega_1(f,R) \sum_{z'\in\Zd} \left\| \nabla^m \Phi(x,t) \right\|_{L^\infty(z'+Q)}. 
\end{align*} 
Since $t\geq 1$, standard estimates for the heat kernel give
\begin{equation*} \label{}
\sum_{z'\in\Zd} \left\| \nabla^m \Phi(x,t) \right\|_{L^\infty(z'+Q)} \leq C(m+1) \int_{\Rd} \left| \nabla^m \Phi \left(x,t - \frac12 \right)  \right|\,dx.
\end{equation*}
Using the bounds 
\begin{equation*}
 \int_{\Rd} \left| \nabla^m \Phi(x,t) \right|\,dx \leq t^{-\frac m2} \left( C (1+m)\right)^{\frac m2},
\end{equation*}
which are proved below in Lemma~\ref{l.heatkernelL1}, we obtain
\begin{equation*} 
\left| \diff_{\mathcal T_1} \nabla^m u(0,t) \right|  \leq C \omega_1(f,R)   t^{-\frac m2} \left( C (1+m)\right)^{\frac m2 +1} .
\end{equation*}
Arguing in a similar way, we also have the crude bound
\begin{equation}
\label{e.crude}
 \sup_{x\in\Rd} \left| \nabla^{m} u(x,t)  \right| \leq C \| f \|  t^{-\frac m2} \left( C (1+m)\right)^{\frac m2+1} .
\end{equation}
Next we observe that since $z_1 \in B_R$, 
\begin{equation*}
\left| \nabla^m u(z_1,t) - \nabla^m u(0,t) \right|  \leq R \sup_{x\in\Rd} \left| \nabla^{m+1} u(x,t)  \right| \leq R \osc_{x\in\Rd}  \nabla^{m+1} u(x,t),
\end{equation*}
where we used that $\langle \nabla^{m+1} u(\cdot,t)\rangle=0$ since $u(\cdot,t)$ is uniformly almost periodic.
Since $y_1\in \R^d$ was arbitrary, the triangle inequality and the three previous inequalities yield
\begin{equation} \label{e.snapback}
\osc_{x\in\Rd}  \nabla^m u(x,t) \leq  C \omega_1(f,R) t^{-\frac m2} \left( C (1+m)\right)^{\frac m2+1}   + CR \osc_{x\in\Rd}  \nabla^{m+1} u(x,t).
\end{equation}
Iterating~\eqref{e.snapback}, we obtain, for every $m\in\N$, using in addition \eqref{e.crude} to close the expression:
\begin{multline*} 
\osc_{x\in\Rd} u(x,t) \\
\leq C \omega_1(f,R)  \sum_{n=0}^{m-1} \left( \frac{C (1+n)R^2}{t}\right)^{\frac n2}(1+n) +C  \|f\| \left( \frac{C (1+m)R^2}{t}\right)^{\frac m2}(1+m)   .
\end{multline*}
We next choose $m$ and $R$ to minimize the expression on the right side. We proceed by first selecting~$m$ to be the largest positive integer such that  
\begin{equation*}
C R^2 t^{-1} (1+m) \leq \frac12,
\end{equation*}
with $m=0$ if no such positive integer exists. 
If $m\neq 0$, this yields for all $0\le n\le m$,
\begin{eqnarray*}
\left( \frac{C (1+n)R^2}{t}\right)^{\frac n2}(1+n)&\le & (1+n)\exp(\frac n2\log \frac{C (1+m)R^2}{t}) \\
&\leq & (1+n) \exp(-\frac n2 \log 2) \le C\exp(-\frac n4 \log 2),
\end{eqnarray*}
so that
\begin{equation*}
 \left( C R^2 t^{-1} (1+m) \right)^{\frac m2}(1+m) \leq C\exp\left( - c R^{-2} t \right) 
\end{equation*}
and
\begin{equation*}
\sum_{n=0}^{m-1} \left( C R^2 t^{-1} (1+n) \right)^{\frac n2} \leq C.
\end{equation*}
For $m\neq 0$ this gives us the bound
\begin{equation*}
\osc_{x\in\Rd} u(x,t)  
 \leq C\omega_1(f,R) +C \| f \| \exp\left( - c R^{-2} t \right),
\end{equation*}
which also trivially holds if $CR^2t^{-1}>1$, that is, for $m=0$.
This completes the proof of~\eqref{e.ergthm1} for k=1. For~\eqref{e.ergthm2} with $k=1$, we proceed similarly: by induction, we get
 \begin{multline*}
\sup_{x\in\Rd} \left| \nabla u(x,t) \right| \\
 \leq R^{-1}\omega_1(f,R) \sum_{n=1}^{m-1} \left( C R^2 t^{-1}(1+n) \right)^{\frac n2}(1+n) +  \left( C R^2 t^{-1} (1+m) \right)^{\frac m2}  (1+m) \| f \|
\end{multline*}
and then notice that the same choice of $m$ also gives 
\begin{equation*}
\sum_{n=1}^{m-1} \left( C R^2 t^{-1} (1+n) \right)^{\frac n2}(1+n) \leq CRt^{-\frac12}.
\end{equation*}
This completes the proof of the proposition in the case $k=1$.

\smallskip

\emph{Step 2.} We now argue that the statement of the proposition for general $k\in\N$ follows by iterating the conclusion for $k=1$ obtained in the first step. We assume that the conclusion holds for some $k\in\N$ and demonstrate it for $k+1$. We fix $t \geq k+1$. 

\smallskip

Let $y\in\Rd$ and choose $z \in B_R$ so that 
\begin{equation*}
\omega_{k}(\diff_{yz} f,R) \leq \omega_{k+1}(f,R).
\end{equation*}
According to the induction hypothesis, we have
\begin{align*}
\osc_{\Rd} \diff_{yz} u \left(\cdot \,,\,\left( \frac {k}{k+1} \right)t\right)
& \leq  C^k \left( \omega_{k}(\diff_{yz} f,R) +  \exp\left( - \frac{c}{kR^2} \left( \frac {k}{k+1} \right)t\right) \| \diff_{yz} f \| \right) \\
& \leq C^{k} \left( \omega_{k+1}(f,R) + 2\exp\left( - \frac{ct}{(k+1)R^2} \right) \| f \| \right).
\end{align*}
It is clear that, by~\eqref{e.falmostper}, $u\left(\cdot, s\right)$ is uniformly almost periodic for each $s>0$ in the sense that $\rho_1(u(\cdot,s),R) \to 0$ as $R\to \infty$. Therefore $u(\cdot,s)$ has a mean value and 
\begin{equation*}
\left\langle \diff_{yz} u(\cdot,s) \right\rangle 
 = \left\langle T_{y} u(\cdot,s)   \right\rangle -  \left\langle T_{z} u(\cdot,s)   \right\rangle 
 =  \left\langle  u(\cdot,s)   \right\rangle -  \left\langle  u(\cdot,s)  \right\rangle 
 = 0.
\end{equation*}
We deduce that $ \| \diff_{yz} u(\cdot,s) \| \leq \osc_{\Rd} \diff_{yz} u(\cdot,s)$ and therefore obtain
\begin{equation*}
\left\| \diff_{yz} u\left(\cdot \,,\,\left( \frac {k}{k+1} \right)t\right) \right\|_{L^\infty(\Rd)}
\leq C^{k} \left( \omega_{k+1}(f,R) + 2\exp\left( - \frac{ct}{(k+1)R^2} \right) \| f \| \right).
\end{equation*}
Taking the supremum over $y\in\Rd$, we get
\begin{equation}
\label{e.dunkslam}
\omega_1\left( u\left(\cdot \,,\,\left( \frac {k}{k+1} \right)t\right),R \right) 
 \leq C^{k} \left( \omega_{k+1}(f,R) + 2\exp\left( - \frac{ct}{(k+1)R^2} \right) \| f \| \right).
\end{equation}
Note that the crude bound~\eqref{e.crude} above implies that, for every $s\geq 1$,
\begin{equation*}
\left\| u(\cdot,s ) \right\|_{L^\infty(\Rd)} \leq C\| f \|.
\end{equation*} 
Using this,~\eqref{e.dunkslam} and the proposition for $k=1$ proved in the previous step, we get, for every $s\geq 1$, 
\begin{align*}
\osc_{\Rd} u\left(\cdot \,,\, \left( \frac {k}{k+1} \right)t+s\right) 
& \leq C \inf_{R\geq 1} \left( \omega_1\left( u\left(\cdot \,,\,\left( \frac {k}{k+1} \right)t\right),R \right) + C\exp\left( - \frac{cs}{R^2} \right) \| f \| \right) \\
& \leq C^{k+1} \inf_{R\geq 1} \bigg( \omega_{k+1}(f,R) + \exp\left( - \frac{ct}{(k+1)R^2} \right) \| f \|\\
& \qquad \qquad+\exp\left(-\frac{cs}{R^2} \right)\|f\|\bigg).
\end{align*}
Applying this to $s=t/(k+1)$, we obtain~\eqref{e.ergthm1} for $k+1$. The proof of~\eqref{e.ergthm2} for $k+1$ is similar. The proposition now follows by induction. 
\end{proof}

To complete the proof of Proposition~\ref{p.ergodicthm}, we must prove the following lemma. This is surely well-known, but since we could not find a reference, we include a complete argument for the convenience of the reader. 

\begin{lemma}
\label{l.heatkernelL1}
For every $n\in\N$ and $t>0$,
\begin{equation*}
 \int_{\Rd} \left| \nabla^n \Phi(x,t) \right|\,dx \leq \left( C t^{-1} (1+n)\right)^{\frac n2}.
\end{equation*}
\end{lemma}
\begin{proof}
By rescaling, it suffices to establish the lemma for $t=1/4$. Denote 
\begin{equation*}
\varphi(x):= \exp\left( -\left| x \right|^2 \right) = \pi^{\frac d2} \Phi\left(x,\frac14\right).
\end{equation*}

\smallskip

\emph{Step 1.} We derive an expression for $\nabla^n\varphi$. Recall that the sequence of Hermite polynomials $\{ H_n(t)\}_{n\in\N}$ is given by the recursion formula
\begin{equation*}
\left\{ 
\begin{aligned}
& H_{n+1} (t) = 2t H_n(t) - 2n H_{n-1}(t), & n\in\N, \\
& H_0(t) = 1,  \quad H_1(t) = t.
\end{aligned} 
\right.
\end{equation*}
They satisfy the expression
\begin{equation*}
H_n(t) = (-1)^n \exp(t^2) \frac{d^n}{dt^n} \left( \exp(-t^2) \right).
\end{equation*}
It follows then that we may express the derivatives of $\varphi$ in terms of $H_n(t)$ by
\begin{equation*}
\left( \frac{\partial}{\partial x_1} \right)^{n_1}\cdots \left( \frac{\partial}{\partial x_d}\right)^{n_d} \varphi(x) = \varphi(x)  \prod_{j=1}^d  H_{n_j}(x_j) .
\end{equation*}
Thus
\begin{equation}
\label{e.nablakhermite}
\left| \nabla^n \phi(x) \right| \leq \varphi(x) \sum_{n_1+\cdots+n_d = n} \ \left|  \prod_{j=1}^d  H_{n_j}(x_j)  \right|.
\end{equation}

\smallskip

\emph{Step 2.} We claim that, for every $m,n\in\N$,
\begin{equation}
\label{e.integralHermite}
I_{m,n}:=\int_{-\infty}^\infty |t|^m \left| H_n(t) \right| \exp\left( -t^2 \right)\,dt \\
 \leq C^n \left( 1+m+n \right)^{\frac{m+n}{2}}.
\end{equation}
By a routine computation, it is easy to check that this estimate holds for $m\in\N$ and $n\in\{0 ,1\}$. Indeed, we have 
\begin{equation*}
I_{2m,0} \leq m! \sqrt{\pi}
\end{equation*}
and $I_{2m,1} = I_{2m+1,0}$. The recursion formula for $H_n(t)$ gives the bound 
\begin{equation*}
I_{m,n+1} \leq 2I_{m+1,n} + 2n I_{m,n-1} \quad \mbox{for all} \ m,n\in\N,
\end{equation*}
and from this we can easily check that the estimate holds for $m\in\N$ and $n=n_0+1$ provided it holds for $m\in\N$ and $n\in \{n_0-1,n_0\}$. We therefore obtain~\eqref{e.integralHermite} by induction. 

\smallskip

\emph{Step 3.} We complete the argument. Using~\eqref{e.nablakhermite} and~\eqref{e.integralHermite} with $m=0$, we compute
\begin{align*}
\int_{\Rd} \left| \nabla^n \varphi(x) \right|\,dx 
& \leq  \sum_{n_1+\cdots n_d=n} d^n \ \prod_{j=1}^d \int_{-\infty}^\infty \left| H_{n_j} (t) \right| \exp\left( -t^2 \right)\,dt \\
& \leq\sum_{n_1+\cdots n_d=n} C^n \left( (1+n \right)^{\frac n2}.
\end{align*}
Since the number of elements in the sum is 
\begin{equation*}
\frac{n!}{d! (n-d)!} \leq n^d \leq (Cd)^n \leq C^n,  
\end{equation*}
the proof is complete. 
\end{proof}

\section{Quantitative weak convergence}
\label{s.proofs}

In this section, we prove Theorem~\ref{t.correctors}. We work with the standard regularization  of the corrector equation
\begin{equation*}
\ep^2 \phi_\ep - \nabla \cdot\left( \a(x)\left( e+\nabla \phi_\ep \right)\right) = 0\quad \mbox{in} \ \Rd. 
\end{equation*}
We call $\phi_\ep$ the \emph{approximate correctors}. We proceed by obtaining estimates on $\phi_\ep$ and its gradient and eventually passing to the limit $\ep \to 0$. The argument has roughly three steps: (i) use the regularity theory developed in~\cite{ASh} to estimate $\omega_k(\nabla \phi_\ep,R)$ in terms of $\rho_k(\a,R)$, thereby transferring quantitative ergodic information from the coefficients directly to the gradients of the approximate correctors; (ii) apply the ergodic theorem in the form 
of \pref{ergodicthm} to get control of spatial averages of $\nabla \phi_\ep$; and (iii) apply a functional inequality (Proposition~\ref{p.heatfuckingflow} below) to obtain pointwise control of $\phi_\ep$ which will be independent of the parameter~$\ep$. 

\smallskip

We set up the proof in the next two subsections by introducing the main ingredients for steps (ii) and (iii). In Section~\ref{ss.phiep} we present the main argument. 

\subsection{Regularity theory}
A regularity theory for periodic homogenization was developed by Avellaneda and Lin~\cite{AL1,AL2} using compactness methods. Their arguments require the existence of a bounded corrector and hence are not applicable in our setting. Recently, a new quantitative argument was introduced by the first author and Smart~\cite{AS}, in the context of random homogenization, which does not require compactness or bounds on correctors. This lead to the development of a regularity theory for homogenization 
in the almost periodic case~\cite{ASh} in addition to the stochastic setting~\cite{AM,GNO2}.

\smallskip

The major ingredients we need from the regularity theory are uniform Lipschitz and Calder\'on-Zygmund estimates proved in~\cite{ASh}. These give us the control on the gradient of the approximate correctors required to apply Proposition~\ref{p.ergodicthm}. The statements here are given in a slightly more general form than what is presented in~\cite{ASh}. The modifications which are needed are explained below in Remark~\ref{r.modifications}. In what follows, $p_*$ is the Sobolev exponent $p_* :=\frac{dp}{d+p}$, defined for any $p\in[1,\infty)$, and $p^*:= \frac{dp}{d-p}$, which is defined for $p<d$. We also take $d^*$ to be any finite exponent and $p^*=\infty$ if $p>d$. 

\smallskip

\begin{proposition}[{\cite[Theorems 4.6 and~5.1]{ASh}}]
\label{p.lipschitz}
Assume~$\a:\Rd \to \R^{d\times d}$ satisfies~\eqref{e.ue} and~\eqref{e.aholder}
and that the modulus $\rho_1$ defined in~\eqref{e.rho} satisfies, for some $M\geq 1$ and $\beta > \frac52$,
\begin{equation*} \label{}
\rho_1(R) \leq M \left( \log R \right)^{-\beta} \quad \mbox{for every} \ R\geq 2.
\end{equation*}
Let  $R \geq 1$, $p\in [2,\infty)$, $\alpha \in (0,1]$, $\mathbf{f}\in L^p(B_R;\Rd)$, $g\in L^{p_*}(B_R)$ and $u\in H^1(B_R)$ be a solution of the equation
\begin{equation*} \label{}
-\nabla \cdot\left( \a(x) \nabla u(x) \right) = g +  \nabla \cdot \mathbf{f}  \quad \mbox{in} \ B_R.
\end{equation*}
Then there exists $C(p,M,\beta,d,\Lambda,\gamma) \geq 1$ such that
\begin{multline} \label{e.CZ}
\left( \fint_{B_{R/2}} \left| \nabla u(x) \right|^p \right)^{\frac1p} \\
\leq C \left( \frac1R \left( \fint_{B_R} \left| u(x) \right|^2\,dx \right)^{\frac12} + \left( \fint_{B_{R}} \left| \mathbf{f}(x) \right|^p \,dx \right)^{\frac1p} + R \left( \fint_{B_{R}} \left| g(x) \right|^{p_*} \,dx\right)^{\frac1{p_*}} \right)
\end{multline}
and, if $\mathbf{f} \in C^{0,\alpha}(B_R)$ and $g \in L^q(B_R)$ for some $q>d$, then, with $C$ depending additionally on $\alpha$ and $q$, we have
\begin{multline} \label{e.Lip}
\sup_{x\in B_{R/2}} \left| \nabla u (x) \right| \\
\leq C \left( \frac 1R \left( \fint_{B_R} \left| u(x) \right|^2\,dx \right)^{\frac12} 
+ R^{\alpha} \left[ \mathbf{f} \right]_{C^{0,\alpha}(B_R)}  
+ R \left( \fint_{B_{R}} \left| g(x) \right|^{q} \,dx \right)^{\frac1{q}} \right).
\end{multline}
\end{proposition}

\begin{remark}
\label{r.modifications}
The estimates in the above proposition were stated and proved in~\cite{ASh} in the case $g=0$. However, this is no less general Proposition~\ref{p.lipschitz}, in view of the fact that any $g\in L^{q}(B_R)$ can be written as 
\begin{equation*}
g = \nabla \cdot \mathbf{g}
\end{equation*}
 where $\mathbf{g} := -\nabla u \in W^{1,q}(B_R)$ and $u \in W^{2,q}(B_R)$ is the solution the Dirichlet problem
 \begin{equation*}
\left\{ 
\begin{aligned}
& -\Delta u = g & \mbox{in} & \ B_R, \\
& u = 0 & \mbox{on} & \ \partial B_R. 
\end{aligned}
\right.
\end{equation*}
Indeed, by the standard Calderon-Zygmund and Sobolev estimates, we have, in the case $q=p_*$, 
\begin{equation*}
 \left( \fint_{B_{R}} \left| \mathbf{g}(x) \right|^p \,dx \right)^{\frac1p} \leq CR \left( \fint_{B_{R}} \left| g(x) \right|^{p_*}\,dx \right)^{\frac1{p_*}} 
\end{equation*}
and, for $q>d$, denoting $\alpha':=1-\frac dq$,
\begin{equation*}
 R^{\alpha'} \left[ \mathbf{g} \right]_{C^{0,\alpha'}(B_R)}
 \leq CR \left( \fint_{B_{R}} \left| g(x) \right|^{q}\,dx \right)^{\frac1{q}}.
\end{equation*}
In both estimates, the constant $C$ depends only on $d$ and $q$. 
Therefore we can absorb the $g$ into the divergence of the vector field. 
\end{remark}

For the remainder of this section, we assume the hypotheses of Theorem~\ref{t.correctors} are in force. We next put the regularity estimates in a form more suitable for our analysis. 

\begin{lemma} \label{l.aux u}
Fix $p\in [2,\infty)$. There exists $C(p,\data)\geq 1$ such that, for every $\ep \in (0,1]$, $\mathbf{f} \in L_{\mathrm{loc}}^p(\Rd;\Rd)$, $g\in L^{p_*}_{\mathrm{loc}}(\Rd)$ and $u \in H_{\rm loc}^1(\Rd)$ satisfying 
\begin{equation} \label{e.aux u}
\ep^2 u  - \nabla \cdot\left( \a(x) \nabla u \right)  = g +  \nabla \cdot \mathbf{f} \quad \mbox{in} \ \R^d ,
\end{equation}
we have the estimates
\begin{multline}
\label{e.auxu}
  \sup_{z' \in \R^d }  \left( \fint_{B_{1/\ep}(z')} \left| \nabla u(x) \right|^{p}  \, dx \right)^{\frac1p}  + \ep \sup_{\R^d } |u| \\ 
  \leq C \sup_{z' \in \R^d } \left(  \left( \fint_{B_{1/\ep}(z')} \left| \mathbf{f}(x) \right|^p \, dx \right)^{\frac1{p}} + \frac1\ep \left(  \fint_{B_{1/\ep}(z')} \left| g(x) \right|^{p_*} \, dx \right)^{\frac1{p_*}}  \right).
\end{multline}
Moreover, in the case that $\mathbf{f} \in C^{0,\alpha}_{\mathrm{loc}} (\R^d;\R^d)$ and $g \in L^q_{\mathrm{loc}}(\R^d)$ for some $\alpha\in (0,1]$ and $q>d$, we have the estimate
\begin{equation}
\label{e.auxunablauinfty}
 \left\| \nabla u \right \|_{L^\infty(\R^d)}   \leq  C\sup_{z' \in \R^d} \left( 
  \ep^{-\alpha} \left[ \mathbf{f} \right]_{C^{0,\alpha}(B_{1/\ep}(z'))} 
  +  \frac1\ep \left(  \fint_{B_{1/\ep}(z')} \left| g(x) \right|^{q} \, dx \right)^{\frac1{q}}  \right),
\end{equation}
with the constant $C$ depending additionally on $\alpha$ and $q$. 
\end{lemma}
\begin{proof}
In view of Remark~\ref{r.modifications}, we can assume $g=0$. Test the equation~\eqref{e.aux u} with $u \psi^2$, where $\psi \in C_0^\infty(\R^d)$ to get
\begin{equation*} 
\frac12 \ep^2 \int_{\R^d} u^2 \psi^2 \, dx + \frac14 \int_{\R^d} |\nabla u|^2 \psi^2 \, dx \\ \leq \frac34  \Lambda^2 \int_{\R^d} u^2 |\nabla \psi|^2 \, dx  +  4 \int_{\R^d} |f|^2 \psi^2 \, dx.
\end{equation*}
Taking first $\psi \equiv \tilde \psi$, $\tilde \psi(x) = \exp(-R^{-1} |x -z'|)$, with $z' \in \R^d$ and $R = 2\Lambda/\ep$, we see that the first term on the right can be absorbed to the first term on the left, and we arrive after straightforward manipulations to the inequality
\begin{equation} \label{e.u initial 000}
\sup_{z ' \in \R^d} \left( \ep^2 \fint_{B_{2/\ep}(z')} u^2 \, dx +  \fint_{B_{2/\ep}(z')} |\nabla u|^2 \, dx \right) 
 \leq C \sup_{z' \in \R^d} \fint_{B_{1/\ep}(z')}|f|^2  \, dx.
\end{equation}
We now fix $z \in \R^d$ and take $1/\ep < R_{j+1} < R_j \leq 2/\ep$ such that $R_{j}-R_{j+1} \leq 2^{-j-2} /\ep$ and $R_0 = 2/\ep$. Denote, in short, $B_{R_j} = B_{R_j}(z)$. Letting $w_j \in H_0^1(B_{R_j})$ solve $-\Delta w_j = \ep^2 u$ in $B_{R_j}$, the equation of $u$ takes the form 
\begin{equation*} 
- \nabla \cdot\left( \a(x) \nabla u \right)  = \nabla \cdot \left( \mathbf{f} + \nabla w_j \right) \quad \mbox{in} \ B_{R_j}\,.
\end{equation*}
By the standard Calderon-Zygmund theory (as in Remark~\ref{r.modifications}) we have that 
\begin{equation} \label{e.w eq 000}
\left( \fint_{B_{R_j}} |\nabla w_j|^p \, dx \right)^{\frac 1p} 
\leq C   \left(   \fint_{B_{R_j}} \left(\ep |u| \right)^{p_*} \, dx \right)^{\frac 1{p_*}}
\end{equation}
for any $p \geq 2$. Combining this with~\eqref{e.u initial 000} 
gives us for $p=2$ that
\begin{equation*} 
\left( \fint_{B_{R_0}} |\nabla w_0|^{2} \, dx \right)^{\frac 1{2}} \leq C \sup_{z' \in \R^d} \left( \fint_{B_{1/\ep}(z')} |\mathbf{f}|^2  \, dx \right)^{\frac12 }.
\end{equation*}
Next, an application of Proposition~\ref{p.lipschitz} together with a covering argument shows that 
\begin{equation}  \label{e.CZ 000}
\left( \fint_{B_{R_{j+1}}} \left| \nabla u \right|^{p}  \, dx \right)^{\frac1p} \\ 
\leq C 2^{j d/p} \left( \ep \left( \fint_{B_{R_j}} \left| u \right|^p\,dx \right)^{\frac1p} + \left( \fint_{B_{R_j}} \left| \mathbf{f} \right|^{p} \, dx  \right)^{\frac1p} \right)\,,
\end{equation}
and we get  by Sobolev-Poincar\'e inequality and H\"older's inequality that
\begin{equation*} 
\left( \fint_{B_{R_{j+1}}} \left|u \right|^{p^*}  \, dx \right)^{\frac1{p^*}}  \leq C 2^{j d/p}  \left(  \left( \fint_{B_{R_j}} \left| u \right|^p\,dx \right)^{\frac1p} + 
\frac1{\ep} \left( \fint_{B_{R}}  |\mathbf{f}|^p \, dx \right)^{\frac1{p}} \right).
\end{equation*}
We can now start iteration with $p \equiv p_j$, where $p_0 = 2$ and $p_j = p_{j-1}^*$ for $j \in \N$. Then, after finitely many steps, we may conclude  by Morrey's inequality that 
\begin{equation*} 
\sup_{B_{R_j}} |u| \leq C  \left(  \left( \fint_{B_{R_0}} \left| u(x) \right|^2 \,dx \right)^{\frac12} + \frac1\ep \left( \fint_{B_{R_0}} |\mathbf{f}|^p \, dx \right)^{\frac1{p}} \right)\,.
\end{equation*}
with any $p > d$. Applying~\eqref{e.u initial 000} once more leads to 
$$
\sup_{B_{R_j}} |u| \leq \frac{C}{\ep} \sup_{z' \in \R^d}  \left( \fint_{B_{1/\ep}(z')} |\mathbf{f}|^p \, dx \right)^{\frac1{p}} 
$$
Inserting this into~\eqref{e.CZ 000} and iterating to reach any arbitrary $p>d$ concludes the proof of~\eqref{e.auxu}. The proof of~\eqref{e.auxunablauinfty} follows from the previous estimate and~\eqref{e.Lip}. 
\end{proof}

\subsection{Multiscale Poincar\'e inequality}\label{s.poinca}
The following functional inequality can be compared to~\cite[Proposition 6.1]{AKM}. It gives an estimate for the oscillation of an arbitrary bounded and Lipschitz function in terms of the spatial averages of its gradient on every scale (and the spatial averages are given in terms of convolution against the heat kernel). 

\begin{proposition}
\label{p.heatfuckingflow}
For every $u\in W^{1,\infty}(\Rd)$,
\begin{equation*}
\osc_{\Rd} u \leq 2 \int_0^\infty \sup_{y\in\Rd} \left| \int_{\Rd}  \nabla u(x) \cdot \nabla \Phi (y-x,t) \,dx \right| \,dt.
\end{equation*}
\end{proposition}
\begin{proof}
Select $g\in L^1(\Rd)$ such that  
\begin{equation}
\label{e.gmeanzero}
\int_{\Rd} g(x) \,dx = 0.
\end{equation}
Consider the solution $w$ of the Cauchy problem
\begin{equation*}
\left\{ \begin{aligned}
& \partial_t w - \Delta w = 0 & \mbox{in} & \ \Rd \times (0,\infty), \\
& w(\cdot,0) = g & \mbox{in} & \ \Rd. 
\end{aligned} \right.
\end{equation*}
Observe that, for every $t>0$, we have $w(\cdot,t) \in C^\infty(\Rd) \cap L^1(\Rd)$. Define
\begin{equation*}
\phi(t):= \int_{\Rd} u(x) w(x,t)\,dx
\end{equation*}
and compute
\begin{align*}
\frac{d}{ds} \phi(s) 
& = \int_{\Rd} u(x) \partial_t w(x,s)\,dx \\
& = \int_{\Rd} u(x) \Delta w(x,s)\,dx \\
& = -\int_{\Rd} \nabla u(x) \cdot \nabla w(x,s)\,dx. 
\end{align*}
Inserting the representation formula
\begin{equation*}
\nabla w(x,t) = \int_{\Rd} g(y) \nabla \Phi(x-y,t)\,dy
\end{equation*}
we obtain 
\begin{equation*}
\frac{d}{ds} \phi(s)  =  \int_{\Rd} g(y) \int_{\Rd} \nabla u(x) \cdot \nabla \Phi(x-y,t)\,dx \, dy
\end{equation*}
and then 
\begin{align*}
\left|\frac{d}{ds} \phi(s)  \right| 
& \leq  \int_{\Rd} \left| g(y) \right| \cdot\left|  \int_{\Rd} \nabla u(x) \cdot \nabla \Phi(x-y,t)\,dx \right| \, dy \\
& \leq \| g \|_{L^1(\Rd)} \, \sup_{y\in\Rd} \left| \int_{\Rd}\nabla u(x) \cdot \nabla \Phi(x-y,t)\,dx \right|.
\end{align*}
By~\eqref{e.gmeanzero}, we have 
\begin{equation*}
\lim_{t\to \infty} \| w(\cdot,t) \|_{L^\infty(\Rd)} = 0. 
\end{equation*}
Thus 
\begin{align*}
\left| \int_{\Rd} u(x) g(x) \,dx \right| 
& = \left| \int_{0}^\infty \frac{d}{ds} \phi(s)\,ds \right| \\
& \leq \| g \|_{L^1(\Rd)} \int_0^\infty \sup_{y\in\Rd} \left| \int_{\Rd}\nabla u(x) \cdot \nabla \Phi(x-y,t)\,dx \right|\,dt.
\end{align*}
Taking the supremum over $g\in L^1(\Rd)$ satisfying~\eqref{e.gmeanzero} and $\| g \|_{L^1(\Rd)} \leq 1$, we get
\begin{equation*}
\osc_{\Rd} u \leq 2 \int_0^\infty \sup_{y\in\Rd} \left| \int_{\Rd}\nabla u(x) \cdot \nabla \Phi(x-y,t)\,dx \right|\,dt.
\end{equation*}
This completes the proof of the proposition. 
\end{proof}

\subsection{The proof of Theorem~\ref{t.correctors}}
\label{ss.phiep}

In this subsection, we study the approximate correctors $\phi_\ep$ which, for a fixed unit vector $e\in\partial B_1$, are the unique bounded and almost periodic solutions of the equation
\begin{equation}
\label{e.appcorrectors}
\ep^2 \phi_\ep - \nabla \cdot \left( \a(x) \left( e + \nabla \phi_\ep \right) \right) = 0 \quad \mbox{in} \ \Rd. 
\end{equation}
The existence and uniqueness of $\phi_\ep$ in the class of bounded and uniformly almost periodic functions is standard. See for instance~\cite[Section 5]{S}. 

\smallskip

We begin with some essentially well-known and basic estimates on the approximate correctors which follow from Proposition~\ref{p.lipschitz}. 

\begin{lemma}
There exists $C(\data)\geq 1$ such that, for every $\ep \in (0,1]$,
\begin{equation} \label{e.phieplip}
\ep \left\| \phi_\ep \right\|_{L^\infty(\Rd)} +  \left\| \nabla \phi_\ep \right\|_{L^\infty(\Rd)}   \leq C.
\end{equation}
\end{lemma}
\begin{proof}
The proof is similar to the proof of Lemma~\ref{l.aux u}, but the gradient bound requires a slight variation in the argument. As in the derivation of~\eqref{e.u initial 000} above, an integration by parts using an exponential cutoff function gives the bound
\begin{equation} \label{e.L2tobegin}
\ep^2 \sup_{z\in\Rd} \fint_{B_{1/\ep}(z)} \left| \phi_\ep(x) \right|^2\,dx \leq C.
\end{equation}
We may therefore apply~\eqref{e.CZ} with $g= -\ep^2\phi_\ep$ to the function $x\mapsto e \cdot x + \phi_\ep(x)$ to obtain, for every $p\in [2,\infty)$, $r \in (0,1]$ and $z\in\Rd$,
\begin{equation*} \label{}
\left( \fint_{B_{r/\ep}(z)} \left|  \nabla \phi_\ep(x)  \right|^p \,dx  \right)^{\frac1p} \leq C \left( 1  +  \ep \left( \fint_{B_{2r/\ep}(z)} \left| \phi_\ep(x) \right|^{p_*}\,dx \right)^{\frac1{p_*}}  \right).
\end{equation*}
The Poincar\'e inequality yields
\begin{equation*} \label{}
\left( \fint_{B_{r/\ep}(z)} \left|  \phi_\ep(x)  \right|^p \,dx  \right)^{\frac1p} \leq C \left( \frac1\ep  +  \left( \fint_{B_{2r/\ep}(z)} \left| \phi_\ep(x) \right|^{p_*}\,dx \right)^{\frac1{p_*}}  \right)
\end{equation*}
Applying this to the sequence defined recursively by $p_1:=2^*$ and $p_{j+1} := p_j^*$ and stopping after a finite number of iterations at the first $j$ for which $p_j \geq d+1$, we obtain
\begin{equation*} \label{}
\left( \fint_{B_{1/(2^{j}\ep)}(z)} \left|  \phi_\ep(x)  \right|^{d+1} \,dx  \right)^{\frac1{d+1}} \leq \frac C\ep. 
\end{equation*}
Now we apply~\eqref{e.Lip} with $g= -\ep^2\phi_\ep$ to the function $x\mapsto e \cdot x + \phi_\ep(x)$ to obtain, for every $z\in\Rd$,
\begin{equation*} \label{}
\left\| \nabla \phi_\ep \right\|_{L^\infty(B_{c/\ep}(z))} \leq C.
\end{equation*}
Taking the supremum over all $z$ yields
\begin{equation*} \label{}
\left\| \nabla \phi_\ep \right\|_{L^\infty(\Rd)} \leq C.
\end{equation*}
The estimate for $\ep \left\| \phi_\ep \right\|_{L^\infty(\Rd)}$ follows from the previous estimate and~\eqref{e.L2tobegin}.
\end{proof}

For future reference, we remark that 
\begin{equation}
\label{e.avgphiepzero}
\langle \phi_\ep \rangle = 0.
\end{equation}
Indeed, we have
\begin{equation*}
\ep^2 \left| \fint_{B_R} \phi_\ep(x) \, dx \right| = \frac1R \left| \fint_{\partial B_R} \a(x) (e + \nabla \phi_\ep(x)) \, dH^{n-1}(x)  \right| \leq \frac{C}{R} \to 0
\quad \mbox{as} \ R \to\infty.
\end{equation*}

\smallskip

The next step is to apply the regularity results in Lemma~\ref{l.aux u} to obtain an almost periodic modulus for the gradients of the approximate correctors. This lemma lies at the heart of the section and is the motivation for the definition of~$\rho_k$ in the introduction. Before formulating it, we define an auxiliary quantity to control $L^p$-norms  obtained via a recursive summation formula for differences. To this end, for $(y, z) \in \R^d \times \R^d$, set 
first $
F_1( \a ,  \left( (y, z) \right) )  :=   \left\|\diff_{yz} \a \right\|_{L^\infty} ,
$
and define then recursively 
\begin{equation} \label{e.F_k}
 F_j( \a ,\mathcal T_j) := \left\|\diff_{\mathcal T_j} \a \right\|_{L^\infty} 
+ \sum_{m=1}^{j-1} \sum_{\zeta \in \mathcal P_{m,j}} 
\left\|\diff_{ \zeta^c (\mathcal T_j)} \a \right\|_{L^\infty} F_{m}(\a ,\zeta(\mathcal T_j))
\end{equation}
for $j$-tuple $\mathcal T_j = \left((y_1,z_1),\ldots,(y_j,z_j) \right)$. 
It is  easy to see by induction that we have a rough bound
\begin{equation} \label{e.F_k-vs-G_k}
 F_k(\a ,\mathcal T_k) \leq 2^k k! G_k(\a ,\mathcal T_k)
\end{equation}
for all $k \in \N$ and $\mathcal T_k \in (\R^d \times \R^d)^k$. Here the quantity $G_k(\a,\mathcal T_k)$ is defined by \eqref{e.G_k} in Section~\ref{s.introduction}. Indeed, to prove~\eqref{e.F_k-vs-G_k}, we assume inductively that the sum in $F_{i}(\a ,\zeta(\mathcal T_k))$ has at most $2^i i!$ terms, which is certainly true for $i=1$. Since $\mathcal P_{m,k}$ consists  of $\binom km$ different partitions, we see by the definition of $G_k(\a,\mathcal T_k)$ and induction assumption that 
$$
\sum_{\zeta \in \mathcal P_{m,k}}  \left\|\diff_{ \zeta^c (\mathcal T_j)} \a \right\|_{L^\infty} F_{m}(\a ,\zeta(\mathcal T_j)) \leq \binom km 2^m m! G_k(\a,\mathcal T_k) \leq 2^m k! G_k(\a,\mathcal T_k)\,,
$$
and thus
$$
 F_k( \a ,\mathcal T_k) \leq \left\|\diff_{\mathcal T_k} \a \right\|_{L^\infty} + \sum_{m=1}^{k-1} 2^m k! G_k(\a,\mathcal T_k) \leq 2^k k! G_k(\a,\mathcal T_k)\,,
$$
showing the induction step. 

\smallskip

The following lemma can be compared to \cite[Lemma~4.1 \& Proposition~4.6]{DG} in the random setting (with $\Delta_{xy}$ replaced by the Glauber derivative).

\begin{lemma}
\label{l.babymod}
Fix $\alpha\in \left(0,\frac12\right]$ and $p\in [1,\infty)$. Then there exists~$C(\alpha,p,\data)\geq 1$ such that, for each $k\in\N$, $\ep \in (0,1]$, $s \in [1,1/\ep]$, $\mathcal T_k = \left(  (y_1,z_1), \ldots, (y_k,z_k) \right) 
\in (\Rd \times \Rd)^k$,
\begin{equation}\label{e.softpermute}
\sup_{z' \in \Rd} \left( \fint_{B_s(z')} \left| \diff_{\mathcal T_k} \nabla \phi_\ep(x) \right|^p\,dx \right)^{\frac1p}  \leq C^k (s \ep)^{-\alpha} F_k\left(\a ,\mathcal T_k\right)\,,
\end{equation}
where $F_k$ is defined in~\eqref{e.F_k}.
\end{lemma}
\begin{proof}
The proof is by induction on $k\in\N$. We fix $\ep\in (0,1]$. We denote
\begin{equation*}
\zeta_j := \diff_{\mathcal T_j} \phi_\ep. 
\end{equation*}
for all $\mathcal T_j = \left(  (y_1,z_1), \ldots, (y_j,z_j) \right) \in (\Rd \times \Rd)^j$ and $j \in \N$. 
In what follows, we make use of the product rule for translations and differences:
\begin{equation}
\label{e.productrule}
\diff_{yz} (f\cdot g) = \diff_{yz} f \cdot T_yg +  \diff_{yz} g \cdot T_zf .
\end{equation}

\smallskip

\emph{Step 1.} The proof for $k=1$. 
Observe that $\zeta_1$ satisfies the equation
\begin{equation} 
\label{e.Delta-phi}
\ep^2 \zeta_1  - \nabla \cdot \left( T_{z_1} \a \nabla \zeta_1 \right) = \nabla \cdot \left(  \left( \diff_{y_1z_1}\a \right) \left( e + T_{y_1}\nabla \phi_\ep \right) \right) \quad \mbox{in} \ \Rd. 
\end{equation}
Applying~\eqref{e.auxu} and~\eqref{e.phieplip}, we obtain, for any $q\in [p\vee2,\infty)$, 
\begin{equation*} \label{}
\sup_{z' \in \R^d}\left( \fint_{B_{1/\ep}(z')}\left| \nabla \zeta_1(x) \right|^q\,dx \right)^{\frac1q} \leq C  
\left\| \diff_{y_1z_1} \a \right\|_{L^\infty(\Rd)} 
,
\end{equation*}
where the constant $C$ depends on~$q$ in addition to $(d,\Lambda)$. We furthermore have
\begin{align*}
\left( \fint_{B_{s}(z')} |\nabla \zeta_1(x)|^p\, dx \right)^{\frac 1p} 
& \leq \left( \frac{|B_{1/\ep}(z')|}{|B_{s}(z')|} \fint_{B_{1/\ep}(z')}|\nabla \zeta_1(x)|^q\, dx \right)^{\frac 1q} \\
& =  \left( \ep s \right)^{-\frac {d}{q}} \left(\fint_{B_{1/\ep}(z')} \left|\nabla \zeta_1(x) \right|^q\,dx\right)^{\frac 1q}\,,
\end{align*}
and hence we obtain 
\begin{equation*} \label{}
\sup_{z' \in \R^d} \left( \fint_{B_{s}(z')}  \left| \nabla \zeta_1(x) \right|^p\,dx \right)^{\frac1p} \leq C\left( \ep s \right)^{-\frac {d}{q}}   \left\| \diff_{y_1z_1} \a \right\|_{L^\infty(\Rd)}.
\end{equation*}
Taking $q = \max\left\{ p, d\alpha^{-1}\right\} \leq C(\alpha,p,d,\Lambda)$ yields the result for $k=1$.

\smallskip

\emph{Step 2.} For readability, we give the proof in the case $k=2$ first before presenting the argument for general $k\in\N$ in the next step.

\smallskip

Iterating the product rule~\eqref{e.productrule} gives
\begin{align}
\label{e.productrule2}
\diff_{y_2z_2} \diff_{y_1z_1} (f\cdot g)
& =  \diff_{y_2z_2}\diff_{y_1z_1} f \cdot T_{y_1+y_2}g +  \diff_{y_2z_2}\diff_{y_1z_1} g \cdot T_{z_1+z_2}f \\
& \qquad  + \diff_{y_2z_2} T_{y_1}g \cdot \diff_{y_1z_1} T_{z_2} f + \diff_{y_2z_2} T_{z_1} f \cdot  \diff_{y_1z_1} T_{y_2} g. \notag
\end{align}
Applying the previous identity to $f=\a$ and $g=( e+ \nabla \phi_\ep)$ and taking the divergence of both sides, we obtain an equation for $\zeta_2$: 
\begin{equation*}
\ep^2 \zeta_2 -\nabla \cdot \left( T_{z_1+z_2}\a \nabla \zeta_2 \right) = \nabla \cdot \mathbf{f}_2 \quad \mbox{in} \ \Rd,
\end{equation*}
where the vector field $\mathbf{f}_2$ is given by
\begin{multline*}
\mathbf{f}_2:= \big( \left( \diff_{y_2z_2}\diff_{y_1z_1} \a \right) \left(e + T_{y_1+y_2}\nabla \phi_\ep \right) 
+  \left( \diff_{y_1z_1}T_{z_2} \a  \right)\left( \diff_{y_2z_2} T_{y_1} \nabla \phi_\ep \right) \\
+ \left( \diff_{y_2z_2}T_{z_1} \a \right)\left( \diff_{y_1z_1} T_{y_2} \nabla \phi_\ep  \right)  \big).
\end{multline*}
By~\eqref{e.phieplip} and the result of the previous step we obtain
\begin{align*} 
& \left( \fint_{B_{1/\ep}(z')} \left| \mathbf{f}_2(x) \right|^q \,dx \right)^{\frac1q} 
\\ & \qquad \leq \left\| \diff_{y_2z_2}\diff_{y_1z_1} \a \right\|_{L^\infty(\Rd)}  \left\| e + T_{y_1+y_2}\nabla \phi_\ep \right\|_{L^\infty(\Rd)}
\\ & \qquad  \quad + \left\| \diff_{y_1z_1} \a \right\|_{L^\infty(\Rd)} \left( \fint_{B_{1/\ep}(z')} \left| \diff_{y_2z_2} T_{y_1} \nabla \phi_\ep(x) \right|^q \,dx \right)^{\frac1q}
\\ & \qquad  \quad  + \left\| \diff_{y_2z_2} \a\right\|_{L^\infty(\Rd)} \left( \fint_{B_{1/\ep}(z')} \left| \diff_{y_1z_1} T_{y_2} \nabla \phi_\ep(x) \right|^q \,dx \right)^{\frac1q}
\\ & \qquad  \leq C \left\| \diff_{y_2z_2}\diff_{y_1z_1} \a \right\|_{L^\infty(\Rd)}  + C\left\| \diff_{y_1z_1} \a \right\|_{L^\infty(\Rd)} \left\| \diff_{y_2z_2} \a\right\|_{L^\infty(\Rd)} \,.
\end{align*}
Applying thus~\eqref{e.auxu} and the previous estimate  we get, for any $q\in [p,\infty)$,
\begin{multline*} \label{}
\sup_{z' \in \R^d} \left( \fint_{B_{1/\ep}(z')}  \left| \nabla \zeta_2(x) \right|^q \,dx \right)^{\frac1q} 
\\ \leq C \left\| \diff_{y_2z_2}\diff_{y_1z_1} \a \right\|_{L^\infty(\Rd)}  + C\left\| \diff_{y_1z_1} \a \right\|_{L^\infty(\Rd)} \left\| \diff_{y_2z_2} \a\right\|_{L^\infty(\Rd)} 
\end{multline*}
As in the previous step, taking $q:= \max\{ p, d\alpha^{-1} \}$ yields~\eqref{e.softpermute} for $k=2$.

\smallskip

\emph{Step 3.}
We now argue for general $k\in\N$. Iterating the product rule~\eqref{e.productrule} $k-1$ times leads to the following identity, which generalizes~\eqref{e.productrule2} to $k\geq 2$:
\begin{equation} \label{e.productrulek}
\diff_{\mathcal T_k} (f\cdot g) 
 =  \sum_{j=0}^{k} \sum_{\zeta \in \mathcal P_{j,k}}  \left( \diff_{\zeta^c(\mathcal T_k)} T_{z_{\zeta_1} +\cdots+ z_{\zeta_{j} }} f \right) \cdot \left( \diff_{\zeta(\mathcal T_k)} T_{y_{\zeta_1^c}+\cdots+y_{\zeta_{k-j}^c } } g \right),
\end{equation}
where we recall the definitions of $\mathcal P_{j,k}$ and $\zeta^c$ from Section~\ref{s.introduction}.
Using the previous identity with $f=\a$ and $g = e+\nabla \phi_\ep$ and taking the divergence of both sides, we obtain the following equation for $\zeta_k$:
\begin{equation}\label{e.eq-zeta_k}
\ep^2 \zeta_k - \nabla \cdot\left(T_{z_1+\cdots+z_k} \a \nabla \zeta_k \right) = \nabla \cdot \mathbf{f}_k \quad \mbox{in} \ \Rd,
\end{equation}
where $\mathbf{f}_k$ is the vector field given by
\begin{multline*}
\mathbf{f}_k :=  \left(\diff_{\mathcal T_k} \a \right) \left( e+ T_{y_{1}+\cdots+y_{k} } \nabla \phi_\ep \right)  \\
+ \sum_{j=1}^{k-1} \sum_{\zeta \in \mathcal P_{j,k}}    \left( \diff_{\zeta(\mathcal T_k)} T_{z_{\zeta_1} +\cdots+ z_{\zeta_{j} }} \a \right) \cdot \left( \diff_{\zeta(\mathcal T_k)} T_{y_{\zeta_1^c}+\cdots+y_{\zeta_{k-j}^c } } \nabla \phi_\ep \right),
\end{multline*}

\smallskip

We now make a strong induction hypothesis and assume 
\begin{equation} \label{e.induction hyp}
\sup_{z' \in \R^d} \left( \fint_{B_{1/\ep}(z')} \left| \diff_{\mathcal T_j}  \nabla \phi_\ep(x) \right|^q \,dx \right)^{\frac1q}
  \leq \bar C^j  F_j (\a,\mathcal T_j)
\end{equation}
for all $j$-tuples $\mathcal T_j = \left( (y_1,z_1),\ldots,(y_j, z_j)\right)$ and for all $j \in \{1,\ldots,k-1 \}$. Indeed, by Step 1 this holds for $k=2$. We then claim that~\eqref{e.induction hyp} continues to hold for $k>2$ as well. To this end, using the strong induction assumption~\eqref{e.induction hyp} for $j \in \{1,\ldots,k-1\}$ together with~\eqref{e.phieplip}, we apply the triangle inequality to obtain 
\begin{align*} 
&  \left( \fint_{B_{1/\ep}(z')} \left| \mathbf{f}_k(x) \right|^q \,dx \right)^{\frac1q}     \leq   C \left\| \diff_{\mathcal T_k} \a \right\|_{L^\infty}
+ \sum_{j=1}^{k-1} \bar C^j  \sum_{\zeta \in \mathcal P_{j,k}}    \left\| \diff_{\zeta^c(\mathcal T_k)}  \a \right\|_{L^\infty}   F_j (\a,\mathcal \zeta(T_j)) \,.
\end{align*}
By the definition of $F_k(\a,\mathcal T_k )$ we have that 
$$
\sum_{j=1}^{k-1}  \sum_{\zeta \in \mathcal P_{j,k}}    \left\| \diff_{\zeta^c(\mathcal T_k)}  \a \right\|_{L^\infty}   F_j (\a,\mathcal \zeta(T_j)) \leq F_k(\a,\mathcal T_k ) \,,
$$
and hence
$$
\sup_{z' \in \R^d} \left( \fint_{B_{1/\ep}(z')} \left| \mathbf{f}_k(x) \right|^q \,dx \right)^{\frac1q} \leq \left(C + \bar C^{k-1} \right) F_k(\a,\mathcal T_k ) \,.
$$
The equation \eqref{e.eq-zeta_k} for $\zeta_k$, together with~\eqref{e.auxu}, implies that 
\begin{multline*}
  \sup_{z' \in \R^d }  \left( \fint_{B_{1/\ep}(z')} \left| \nabla \zeta_k (x) \right|^{q}  \, dx \right)^{\frac1q}   \leq C \sup_{z' \in \R^d } \left( \fint_{B_{1/\ep}(z')}  \left| \mathbf{f}_k(x) \right|^q \, dx \right)^{\frac1{q}} \\ \leq C\left(C + \bar C^{k-1} \right) F_k(\a,\mathcal T_k ) \,.
\end{multline*}
Thus we may take $\bar C = 2C$, where $C$ is as in the above inequality. Since $\mathcal T_k$ was arbitrary, we have proven the induction step. We finish the proof 
arguing for $s \in [1,1/\ep]$ similarly as in Step 1. 
\end{proof}

Applying $\sup_{y_k\in \Rd} \inf_{z_k\in B_R} \cdots\sup_{y_1\in \Rd} \inf_{z_1\in B_R}$ to both sides of~\eqref{e.softpermute} and recalling~\eqref{e.F_k-vs-G_k}, yields, for every $\alpha > 0$, 
\begin{equation} \label{e.modcontrol}
\omega_k\! \left( \nabla\phi_\ep,R\right) \leq C^k k! \ep^{-\alpha} \rho_k( \a,R),
\end{equation}
with the constant $C$ depending only on $(\alpha,d,\Lambda)$.

\smallskip

We next combine~\eqref{e.modcontrol} with Propositions~\ref{p.ergodicthm}  and~\ref{p.heatfuckingflow} to deduce that the approximate correctors are nearly bounded pointwise, independently of~$\ep$, and that mesoscopic spatial averages are small. 

\begin{lemma}
\label{l.babybound}
Fix $\alpha > 0$. Then there exist a constant $C(\alpha,\data)\geq 1$ and some time $t_*(data)\sim 1$ such that, for every $\ep \in (0,1]$
and $t\ge t_*$,
\begin{equation} \label{e.babybound}
\sup_{x\in\Rd} \left| \phi_\ep(x) \right|  \leq C \ep^{-\alpha},
\end{equation}
\begin{equation} \label{e.dashiiit}
 \sup_{y\in\Rd} \left| \int_{\Rd} \nabla \phi_\ep(x) \cdot \nabla \Phi(y-x,t)\,dx \right| \leq C \ep^{-\alpha} t^{-1-\frac{\delta}{2}} \left| \log t \right|^{\frac{1+\delta}{2}} + t^{-100} \,,
\end{equation}
and
\begin{equation} \label{e.babyaverage}
\sup_{x\in\Rd} \left| \int_{\Rd} \phi_\ep(y) \Phi\left(x-y,\ep^{-\alpha}\right)\, dy \right| \leq C \ep^{\frac{\alpha\delta}{3}},
\end{equation}
where $\delta >0$ is as in \eqref{e.ergodicassump}.
\end{lemma}
\begin{proof}
Fix $\alpha>0$ and $\ep\in (0,1]$ and denote
\begin{equation*} \label{}
u(x,t):=  \int_{\Rd}  \phi_\ep(y) \Phi\left(x-y,t\right)\, dy.
\end{equation*}
According to Proposition~\ref{p.heatfuckingflow}, for every $t_0\geq 1$, we have by the semi-group property of the heat kernel
\begin{align*} \label{}
\osc_{\Rd} u(\cdot,t_0) 
& \leq \int_0^\infty \sup_{y\in\Rd} \left| \int_{\Rd} \nabla u(x,t_0) \cdot \nabla \Phi(y-x,t)\,dx \right| \,dt  \\
& = \int_0^\infty \sup_{y\in\Rd} \left| \int_{\Rd}  \nabla \phi_\ep(x) \cdot \nabla\Phi(y-x,t+t_0)\,dx \right| \,dt.
\end{align*}
According to the second conclusion \eqref{e.ergthm2} of Proposition~\ref{p.ergodicthm},~\eqref{e.phieplip} and~\eqref{e.modcontrol} (with $\alpha\delta/15$ in place of $\alpha$), we obtain, for every $k\in\N$ and $t\geq k$, 
\begin{align*}
& \sup_{y\in\Rd} \left| \int_{\Rd} \nabla \phi_\ep(x) \cdot \nabla\Phi(y-x,t)\,dx \right| 
 \leq \sup_{y\in\Rd} \left| \nabla \int_{\Rd} \nabla \phi_\ep(x)\Phi(y-x,t)\,dx \right| 
\\ & \qquad  \leq C^{k} \inf_{R\geq 1} \left(  t^{-\frac12} \omega_{k}(\nabla \phi_\ep,R) +  \exp\left( - \frac{ct}{kR^2} \right) \left\| \nabla \phi_\ep\right\|_{L^\infty(\R^d)}\right) 
\\ & \qquad \leq  \inf_{R\geq 1} \left( t^{-\frac12} \ep^{-\alpha\delta/15} C^k k! \rho_k(\a,R) +  C^k \exp\left( - \frac{ct}{kR^2} \right) \right).
\end{align*}
Set $t_*=\inf\{t\ge 2:t\log(Ct)\ge \frac{c}{100}\}$.
For each $t\ge t_*$ and for each $k$ in the range $\N\cap [1,\frac{1}{10}c^{\frac 12}t^{\frac 12}|\log(Ct)|^{-\frac12}]$, which 
implies $t\ge k$ by definition of $t_*$, we estimate the infimum on the last line above by taking 
$R=\frac{1}{10}c^{\frac12}t^{\frac 12}|\log(Ct)|^{-\frac12}k^{-1} \ge 1$.
This yields
\begin{multline}
 \sup_{y\in\Rd} \left| \int_{\Rd} \nabla \phi_\ep(x) \cdot \nabla \Phi(y-x,t)\,dx \right| \\
\leq t^{-\frac12} \ep^{-\alpha\delta/15}  C^k k!  \rho_k\left(\a, \frac{1}{10}c^{\frac12}t^{\frac 12}|\log(Ct)|^{-\frac12}k^{-1}\right) + t^{-100}.
\end{multline}
Taking the infimum over all $k\in \N\cap [1,\frac{1}{10}c^{\frac 12}t^{\frac 12}|\log(Ct)|^{-\frac12}]$ (for which we thus have $t \geq k$) and using the definition of $\rho_*$ and assumption~\eqref{e.ergodicassump}, we get, for every $t\geq t_*$, 
\begin{align*} 
\lefteqn{
 \sup_{y\in\Rd} \left| \int_{\Rd} \nabla \phi_\ep(x) \cdot \nabla \Phi(y-x,t)\,dx \right| 
 } \qquad & \\
 & \leq C\ep^{-\alpha\delta/15}t^{-\frac12} \rho_*\left(\a,\frac{1}{10}c^{\frac12}t^{\frac 12}|\log(Ct)|^{-\frac12} \right) + t^{-100}
\\ & \leq C K \ep^{-\alpha\delta/15} t^{-1-\frac{\delta}{2}} \left| \log t \right|^{\frac12 + \frac{\delta}{2}} + t^{-100} \,.
\end{align*}
Combining these and integrating, using also the Lipschitz bound~\eqref{e.phieplip}, we get
for all $t_0\ge t_*$
\begin{align}
\label{e.t0bound}
\osc_{\Rd} u(\cdot,t_0)
& \leq  \int_0^\infty \sup_{y\in\Rd} \left| \int_{\Rd} \nabla u(x,t_0) \cdot \nabla \Phi(y-x,t+t_0)\,dx \right| \,dt 
  \\
& \leq  C\ep^{-\alpha\delta/15} \int_0^\infty \left(  (t+t_0)^{-1-\frac{\delta}{2}} \left| \log (t+t_0) \right|^{\frac12 + \frac{\delta}{2}} + (t+t_0)^{-100}  \right)\,dt \notag \\
& \leq C \ep^{-\alpha\delta/15} t_0^{-2\delta/5}.  \notag
\end{align}
Recalling from~\eqref{e.avgphiepzero} that $\left\langle \phi_\ep \right\rangle = 0$, we may take $t_0 = \ep^{-\alpha}$ to obtain~\eqref{e.babyaverage}.

\smallskip 

To obtain the first conclusion~\eqref{e.babybound}, we use~\eqref{e.phieplip} and~\eqref{e.t0bound} with $t_0 =t_*\sim 1$ and the fact that $\alpha\delta/15 \leq \alpha$ to get, for each $x\in\Rd$, 
\begin{align*}
\lefteqn{
\left| \int_{B_R(x)} \phi_\ep(z) \Phi(z-x,t_*)\,dz \right| 
} \qquad & \\
& \leq \left| \int_{\Rd} \phi_\ep(z) \Phi(z-x,t_*)\,dz \right| + C\ep^{-1} \left| \int_{\Rd\setminus B_R(x)}\Phi(z-x,t_*)\,dz \right| \\
& \leq C \ep^{-\alpha}  + C\ep^{-1} \exp\left( -c R^{2} \right).
\end{align*}
Taking $R := C \left| 1+ \log \ep\right|^{\frac12}$ yields
\begin{equation*}
\left| \int_{B_R(x)} \phi_\ep(z) \Phi(z-x,t_*)\,dz \right| \leq C \ep^{-\alpha}. 
\end{equation*}
Since $\frac12 \leq \int_{B_R(x)} \Phi(z-x,t_*)\,dz \leq 1$, we may again use the Lipschitz bound~\eqref{e.phieplip}:
\begin{align*}
|\phi_\ep(x)| & \leq 2 \left| \int_{B_R(x)} \phi_\ep(x) \Phi(z-x,t_*)\,dz \right| \\
& \leq CR+\left| \int_{B_R(x)} \phi_\ep(z) \Phi(z-x,t_*)\,dz \right| \\
& \leq C \left| 1+ \log \ep\right|^{\frac12} + C \ep^{-\alpha} 
\\ & \leq C \ep^{-\alpha}.
\end{align*}
This completes the proof. 
\end{proof}

The next step is to obtain an almost periodic modulus for $\phi_\ep$ itself. 

\begin{lemma}
\label{l.phiepmod}
Fix $\alpha\in \left(0,\frac12\right]$. There exists $C(\alpha,\data)\geq 1$ such that, for every $\ep \in (0,1]$, $k \in \N$, and $\mathcal T_k \in (\R^d \times \R^d)^k$,
\begin{equation*} \label{}
\left\| \diff_{\mathcal T_k} \phi_\ep \right\|_{L^\infty(\R^d)} \leq C^k \ep^{-\alpha} F_k(\a,\mathcal T_k) + C  \ep^{\frac{\alpha\delta}{6}\wedge 1}. 
\end{equation*}
\end{lemma}
\begin{proof}
Fix~$\alpha\in \left(0,\frac12\right]$ and $x_0\in\Rd$. Applying Jensen's inequality together with~\eqref{e.softpermute} with $s=1/\ep$, we obtain, for every $1\leq q< p < \infty$,
\begin{align*} 
\left( \fint_{B_{\ep^{-\alpha/2}}(x_0)} \left| \diff_{\mathcal T_k} \nabla\phi_\ep(x) \right|^q \,dx\right)^{\frac1q} 
& \leq  \left( \fint_{B_{\ep^{-\alpha/2}}(x_0)} \left| \diff_{\mathcal T_k} \nabla\phi_\ep(x) \right|^p \,dx\right)^{\frac1p} 
\\ & \leq \left( \frac{B_{1/\ep}(x_0)}{B_{\ep^{-\alpha/2}}(x_0)} \fint_{B_{1/\ep}(x_0)} \left| \diff_{\mathcal T_k} \nabla\phi_\ep(x) \right|^p \,dx\right)^{\frac1p} 
\\ & \leq 
C^k  \ep^{-\frac{d(1-\alpha/2)}{p}} F_k(\a,\mathcal T_k).
\end{align*}
Taking $q=d+1$ and $p = \max\{d+1 , 2 d(1-\alpha/2)/\alpha\}$ so that $\frac{d(1-\alpha/2)}{p}\leq \frac{\alpha}{2}$, we get
\begin{equation*} \label{}
 \left( \fint_{B_{\ep^{-\alpha/2}}(x_0)} \left|\diff_{\mathcal T_k} \nabla \phi_\ep(x) \right|^{d+1} \,dx\right)^{\frac1{d+1}} \leq C^k  \ep^{-\frac{\alpha}{2}} F_k(\a,\mathcal T_k).
\end{equation*}
By Morrey's inequality, 
\begin{equation} \label{e.oscestdiff}
 \osc_{B_{\ep^{-\alpha/2}}(x_0)} \diff_{\mathcal T_k} \phi_\ep \leq C^k  \ep^{-\alpha} F_k(\a,\mathcal T_k) . 
\end{equation}
To turn the estimate of the oscillation into one for the supremum, we use 
\begin{align*} \label{}
 \sup_{B_{\ep^{-\alpha/2}}(x_0)} \left| \diff_{\mathcal T_k} \phi_\ep \right| 
 & \leq  \osc_{B_{\ep^{-\alpha/2}}(x_0)} \diff_{\mathcal T_k} \phi_\ep + \frac{ \left| \int_{B_{\ep^{-\alpha/2}}(x_0)}\diff_{\mathcal T_k} \phi_\ep(y) \Phi(x_0-y,\ep^{-\alpha/2})\,dy \right| }{ \int_{B_{\ep^{-\alpha/2}}} \Phi(y,\ep^{-\alpha/2})\,dy}\,. 
\end{align*}
The second term on the right side may be estimated using~\eqref{e.babyaverage} together with
\begin{equation*} \label{}
 \int_{B_{\ep^{-\alpha/2}}} \Phi(y,\ep^{-\alpha/2})\,dy \geq c, 
\end{equation*}
and
\begin{align*} \label{}
\int_{\Rd \setminus B_{\ep^{-\alpha/2}}(x_0)} \left| \phi_\ep(y) \right| \Phi(x_0-y,\ep^{-\alpha/2})\,dy 
& \leq C\ep^{-1} \int_{\Rd \setminus B_{\ep^{-\alpha/2}}(x_0)} \Phi(x_0-y,\ep^{-\alpha/2})\,dy \\
&\leq C \ep^{100},
\end{align*}
and combining then the result of this calculation with~\eqref{e.oscestdiff} we get
\begin{equation*} \label{}
 \sup_{B_{\ep^{-\alpha/2}}(x_0)} \left| \diff_{\mathcal T_k} \phi_\ep \right|  \leq  C^k  \ep^{-\alpha} F_k(\a,\mathcal T_k)  + C \ep^{\alpha\delta/6\wedge 100}.
\end{equation*}
Taking the supremum over $x_0\in\Rd$, we get 
\begin{equation*} \label{}
\sup_{\Rd} \left| \diff_{\mathcal T_k} \phi_\ep \right| \leq C^k  \ep^{-\alpha} F_k(\a,\mathcal T_k) + C  \ep^{\alpha\delta/6 \wedge 100}\,. \qedhere
\end{equation*}
\end{proof}

We now turn our attention to the differences of the approximate correctors at two successive dyadic scales. We introduce the function
\begin{equation*} \label{}
\psi_\ep(x) := \phi_\ep(x) - \phi_{2\ep}(x).
\end{equation*}
Observe that $\psi_\ep$ is the solution of the equation
\begin{equation} \label{e.psieq}
\ep^2 \psi_\ep - \nabla \cdot \left( \a(x) \nabla \psi_\ep \right) = 3\ep^2 \phi_{2\ep} \quad \mbox{in} \ \Rd.
\end{equation}
We next apply Propositions~\ref{p.heatfuckingflow} and~\ref{p.ergodicthm} to get pointwise estimates on $\psi_\ep$. 

\begin{lemma}
\label{e.psibounds}
There exists
$C(\data)\geq 1$ such that, for every $\ep \in (0,1]$,
\begin{equation} \label{e.psibound}
\sup_{x\in\Rd} \left| \psi_\ep(x)  \right|  \leq C\ep^{\frac{\delta}{8}\wedge \frac14}.
\end{equation}
\end{lemma}
\begin{proof}
Observe that, for every $\alpha\in \left(0,\frac14\right]$, there exists $C(\alpha,\data,d,\Lambda)\geq 1$ such that, for every $\ep \in (0,1]$,
\begin{equation} \label{e.psi bound 1}
\sup_{x\in\Rd} \left| \nabla \psi_\ep(x)  \right|  \leq C\ep^{1-\alpha}
\end{equation}
Indeed, in view of the equation for~$\psi_\ep$ and~\eqref{e.babybound}, this is immediate from~\eqref{e.auxunablauinfty} in Lemma~\ref{l.aux u}.

\smallskip

By Proposition~\ref{p.heatfuckingflow} and the fact that $\langle \psi_\ep \rangle = 0$,
\begin{equation}
\label{e.startingpoint}
\sup_{x\in\Rd} \left| \psi_\ep(x) \right| \leq \osc_{x\in\Rd} \psi_\ep(x) \leq 2 \int_0^{\infty} \sup_{y\in\Rd} \left| \int_{\Rd} \nabla \psi_\ep(x) \cdot \nabla \Phi(y-x,t)\,dx \right|\,dt.
\end{equation}
The gradient bound in~\eqref{e.psi bound 1} and Lemma~\ref{l.heatkernelL1} give, for every $t>0$, 
\begin{align*} 
& \sup_{y\in\Rd} \left| \int_{\Rd} \nabla \psi_\ep(x) \cdot \nabla \Phi(y-x,t)\,dx \right| 
\\ & \qquad \leq C \left\| \nabla \psi_\ep \right\|_{L^\infty(\R^d)}   \int_{\Rd} \left| \nabla \Phi(x,t)\right|  \,dx  
\\ & \qquad \leq C \ep^{1-\alpha} t^{-1/2}\,.
\end{align*}
Thus, for $t_\ep := \ep^{4 \alpha - 2}$ with $\alpha \in (0,1/4]$, we get 
$$
\int_0^{t_\ep} \sup_{y\in\Rd} \left| \int_{\Rd} \nabla \psi_\ep(x) \cdot \nabla \Phi(y-x,t)\,dx \right|\,dt \leq C \ep^{1-\alpha} t_\ep^{1/2} = C \ep^{\alpha}\,.
$$
For the large scales we may use the triangle inequality and~\eqref{e.dashiiit} as follows:
$$
\sup_{y\in\Rd} \left| \int_{\Rd} \nabla \phi_\ep(x) \cdot \nabla \Phi(y-x,t)\,dx \right|  \leq  C \ep^{-\alpha} t^{-1-\frac{\delta}{2}} \left| \log t \right|^{\frac{1+\delta}{2}} + t^{-100} \,.
$$
A similar bound naturally  holds for $\phi_{2\ep}$.  Thus we get
\begin{multline*}
\int_{\ep^{-1}}^\infty \sup_{y\in\Rd} \left| \int_{\Rd} \nabla \psi_\ep(x) \cdot \nabla \Phi(y-x,t)\,dx \right| \, dt 
\\ \leq C\ep^{-\alpha} \int_{\ep^{-1}}^\infty \left( t^{-1-\frac{\delta}{2}} \left| \log t \right|^{\frac{1+\delta}{2}} + t^{-100}  \right) \, dt \leq C \ep^{-\alpha + \frac{\delta}{4}}\,.
\end{multline*}
Combining the estimates above and choosing $\alpha = \frac{\delta}{8}\wedge \frac14$ finishes the proof.
\end{proof}

We now complete the proof of Theorem~\ref{t.correctors} by summing the previous lemma over dyadic scales. 

\begin{proof}[{Proof of Theorem~\ref{t.correctors}}]
We may assume that $\delta \leq 1$. 
According to Lemma~\ref{e.psibounds}, for each $n,m\in\N$, we have 
\begin{equation*}
\sup_{x\in\Rd} \left| \phi_{2^{-n-m}} (x) - \phi_{2^{-n}}(x) \right| \leq \sum_{k=n+1}^{n+m} \sup_{x\in\Rd} \left| \psi_{2^{-k}}(x) \right| \leq C\sum_{k=n+1}^{\infty} 2^{-\frac{k \delta}{8}} \leq C 2^{-\frac{n \delta}{8}} .
\end{equation*}
Similarly, by~\eqref{e.psi bound 1},
\begin{equation*}
\sup_{x\in\Rd} \left| \nabla \phi_{2^{-n-m}} (x)- \nabla \phi_{2^{-n}}(x) \right|
\leq \sum_{k=n+1}^{n+m} \sup_{x\in\Rd} \left|\nabla \psi_{2^{-k}}(x) \right| \leq  C \sum_{k=n+1}^{n+m} 2^{-\frac k2} \leq C2^{-\frac n2}.
\end{equation*}
Thus $\left\{ \phi_{2^{-k}} \right\}_{k\in\N}$ is Cauchy in $W^{1,\infty}(\Rd)$ and there exists $\phi\in W^{1,\infty}(\Rd)$ such that 
\begin{equation*}
\left\| \phi - \phi_{2^{-n}}\right\|_{W^{1,\infty}(\Rd)} \leq C2^{-\frac{n \delta}{8}\wedge \frac n2} \to 0 \quad \mbox{as} \ n\to \infty. 
\end{equation*}
It follows immediately that $\phi$ satisfies~\eqref{e.boundedcorrector}. 

\smallskip 

To prove that $\phi$ is uniformly almost periodic, it suffices to notice that it is the uniform limit of uniformly almost periodic functions. Alternatively, we display a quantitative proof which gives more information. 
Applying the bound
\begin{equation*} \label{}
\| \phi- \phi_{2^{-n}} \|_{L^\infty(\Rd)} \leq C2^{-\frac{n\delta}8},
\end{equation*}
(which is proved above) and then Lemma~\ref{l.phiepmod} with $\alpha = \frac12$, we get
\begin{align*}
\rho_1(\phi,R) & \leq \rho_1(\phi_\ep,R) + C2^{-\frac{n\delta}{8}} \\
& \leq C2^{n\alpha} \rho_1(\a,R) + C2^{-\frac{n\alpha\delta}6} + C2^{-\frac{n\delta}{8}} \\
& \leq C2^{\frac n2} \rho_1(\a,R)+ C2^{-\frac{n\delta}{12}}.
\end{align*}
Taking the infimum over $n\in\N$, we get
\begin{equation*}
\rho_1(\phi,R) \leq C\inf_{n\in\N} \left( 2^{\frac {n}2} \rho_1(\a,R) + 2^{-\frac{n\delta}{12}} \right). 
\end{equation*}
The limit as $R\to\infty$ of the right side is clearly zero. Moreover, if $\rho_1(\a,R)$ has power-like decay in $R$, then so does $\rho_1(\phi,R)$. This completes the argument. 
\end{proof}

\section{Examples}
\label{s.examples}

In this section, we give some examples of almost periodic coefficient fields which satisfy the quantitative ergodic assumptions~\eqref{e.ergodicassump} and~\eqref{e.technicalassump}, starting with a discussion of quasiperiodic functions. 

\smallskip

Recall that a quasiperiodic function $f:\Rd \to \R$ is one that takes the form
\begin{equation*} \label{}
f(x) = F(M(x))
\end{equation*}
for some $n\in\N$, a 1-periodic $F:\R^m \to \R$ and linear map $M:\Rd\to \R^m$. We will not distinguish between $M$ and the matrix $M=(m_{ij})$ which gives the linear map. 
We refer to $M$ as the \emph{winding matrix} and the function $F$ as the \emph{lifted function}. 

\smallskip

Note that $\rho_1(f,R)$ can be bounded using only information about $\| \nabla F \|_{L^\infty(\R^m)}$ and $M$. Indeed, it is clear by the periodicity of $F$ that, for every $R\geq 1$ and $y,z\in\Rd$, 
\begin{multline} \label{e.obvious}
\left\| \diff_{yz}  f \right\|_{L^\infty(\Rd)}  \\
 \leq \| \nabla F \|_{L^\infty(\R^m)} \max_{i=1,\ldots,m} \left| \left((My)_i \!\mod 1\right) - \left((Mz)_i \!\mod 1\right) \right|.
\end{multline}
Taking the infimum over $z\in B_R$ and then the supremum over $y\in\Rd$ to get
\begin{equation*} \label{}
\rho_1(f,R) \leq \| \nabla F \|_{L^\infty(\R^m)} \sup_{y\in\Rd} \inf_{z\in B_R} \max_{i=1,\ldots,m} \left| \left((My)_i \!\mod 1\right) - \left((Mz)_i \!\mod 1\right) \right|.
\end{equation*}
The second factor on the right side above depends only on $M$; we denote it by
\begin{equation*} \label{}
\sigma(M,R):= \sup_{y\in\Rd} \inf_{z\in B_R} \max_{i=1,\ldots,m}  \left((My-Mz)_i \!\mod 1 \right).
\end{equation*}
We henceforth also write $\| My -Mz \|_\infty:=  \max_{i=1,\ldots,m}  \left((My-Mz)_i \!\mod 1 \right)$ for short. 
Since $\diff_{yz}f$ is also quasiperiodic with the same winding matrix $M$ and with lifted function $\diff_{My,Mz}F$, we find that 
\begin{equation*} \label{}
\left\| \diff_{y_2z_2}\diff_{y_1z_1} f \right\|_{L^\infty(\Rd) } \leq \left\| \nabla (\diff_{My_1,Mz_1} F) \right\|_{L^\infty(\R^m)} \| My_2 -Mz_1 \|_\infty. 
\end{equation*}
Finally, applying~\eqref{e.obvious} to $\nabla F$ gives that  
\begin{equation*} \label{}
\left\| \nabla (\diff_{My_1,Mz_1} F) \right\|_{L^\infty(\R^m)} \leq \left\| \nabla \nabla F \right\|_{L^\infty(\R^m)} \| My_1 -Mz_1 \|_\infty. 
\end{equation*}
Putting these together, we get 
\begin{equation*} \label{}
\left\| \diff_{y_2z_2}\diff_{y_1z_1} f \right\|_{L^\infty(\Rd) } \leq  \left\| \nabla \nabla F \right\|_{L^\infty(\R^m)}\| My_1 -Mz_1 \|_\infty \| My_2 -Mz_2 \|_\infty
\end{equation*}
and using this together with~\eqref{e.obvious} again, we get
\begin{align*} \label{}
\rho_2(f,R) & = \sup_{y_1\in\Rd} \inf_{z_1\in B_R} \sup_{y_2\in\Rd} \inf_{z_2\in B_R} \max \left\{ \left\|\diff_{y_2z_2} \diff_{y_1z_1} f \right\| ,\, \left\|\diff_{y_1z_1}f \right\| \left\| \diff_{y_2 z_2} f \right\|   \right\} \\
& \leq \sup_{y_1\in\Rd} \inf_{z_1\in B_R} \sup_{y_2\in\Rd} \inf_{z_2\in B_R} \max \big\{  \\
& \qquad \qquad  \left\| \nabla \nabla F \right\|_{L^\infty(\R^m)}\| My_1 -Mz_1 \|_\infty \| My_2 -Mz_2 \|_\infty,  \\
& \qquad \qquad   \left\| \nabla F \right\|_{L^\infty(\R^m)}^2\| My_1 -Mz_1 \|_\infty \| My_2 -Mz_2 \|_\infty
\big\}  \\
& = \max\left\{ \left\| \nabla \nabla F \right\|_{L^\infty(\R^m)}, \left\| \nabla  F \right\|_{L^\infty(\R^m)}^2 \right\} \sigma^2(M,R).
\end{align*}
This argument can obviously be continued to an arbitrary order and we obtain, for every $k\in\N$, 
\begin{equation} \label{e.fantastic}
\rho_k(f,R)  
\leq  \sigma^k(M,R) \max_{n_1,\ldots,n_k}\prod_{j=1}^k \left\| \nabla^jF \right\|_{L^\infty(\Rd)}^{n_j},
\end{equation}
where the maximum is over all $n_1,\ldots,n_k\in\N$ with $k=\sum_{j=1}^{k} j n_j$. 

\smallskip

This reduces the task of bounding $\rho_k(f,R)$ to one of estimating $\sigma(M,R)$, which lies in the realm of discrepancy theory~\cite{DT}. It is here that the Diophantine condition for the matrix $M$ appears naturally.

\smallskip

Given an exponent $\theta > 0$ and a constant $A>0$, we say that $M$ satisfies $\mathcal{D}(\theta,A)$, and we write $M \in \mathcal{D}(\theta,A)$, if
\begin{equation} \label{e.diophantine}
\left| e_i\cdot M^t z \right| \geq A \left| z \right|^{-\theta}, \quad \forall i\in\{ 1,\ldots,d \}, \ \ z\in \Z^m \setminus \{ 0 \}.
\end{equation}
The following proposition, which bounds $\sigma(M,R)$ for $M\in \mathcal{D}(\theta,A)$, is proved in~\cite[Section 8]{S} using the Erd\"os-Turan-Koksma inequality. The argument actually follows well-known methods in discrepancy theory, it is very close to the proof of~\cite[Theorem 1.80]{DT}.

\begin{proposition}
\label{p.crazyerdosshit}
If \eqref{e.diophantine} holds, then there exists a universal $C < \infty$ such that, for all $R\geq 1$ and $M\in \mathcal{D}(\theta,A)$, 
\begin{equation} 
\label{e.sigmabound}
\sigma(M,R) \leq CA^{-\frac1m} R^{-\frac{1}{m(\theta+1)}} \left( \log R \right)^{\frac{m-1}{m}}. 
\end{equation}
\end{proposition}

In~\cite{S}, the constant $C$ in Proposition~\ref{p.crazyerdosshit} is allowed to depend on $(m,\theta)$ and the dependence on $A$ is not made explicit. However, an inspection of the proofs there and of~\cite[Theorem 1.80]{DT}  (using in particular the exponential dependence in $m$ of the constant in the Erd\"os-Turan-Koksma inequality) one arrives at the result as stated above. 

\smallskip

As a consequence of~\eqref{e.fantastic} and Proposition~\ref{p.crazyerdosshit}, if the lifted function $F$ belongs to $C^k(\Rd)$ and $M \in \mathcal{D}(A,\theta)$, then we obtain the estimate 
\begin{equation*} \label{}
\rho_k(f,R) \leq C\left( A,k, \left\| F \right\|_{C^k(\R^m)} \right) R^{-\frac{k}{m(\theta+1)}} \left( \log R\right)^{\frac{k(m-1)}m}.
\end{equation*}
Thus if $F\in C^n(\R^m)$ for some $n> m(\theta+1)$, we obtain, for any $\delta < \delta_0:= n/m(\theta+1) -1>0$,
\begin{equation*}
\rho_*(f,R) \leq C^n n! \rho_n(f,R) \leq  C\!\left(n,\| F \|_{C^n(\R^m)} \right) R^{-1-\delta}.
\end{equation*}
This confirms~\eqref{e.ergodicassump} for Kozlov's class of quasiperiodic coefficients and gives a more precise estimate of the degree of smoothness required for $F$. Obviously~\eqref{e.fantastic} and~\eqref{e.sigmabound} for $n=1$ also give~\eqref{e.technicalassump}. If in fact $F\in C^\infty(\Rd)$, then we obtain that $\rho_*(f,\cdot)$ decays faster than any power: for every $s\geq 1$, 
\begin{equation*} \label{}
\rho_*(f,R) \leq C\left( s, \left\| F \right\|_{C^{\lceil sm(\theta+1) \rceil}(\R^m)} \right) R^{-s}.
\end{equation*}

\smallskip

We next generalize the above discussion to the case $m=\infty$ and thereby obtain conditions on which almost periodic coefficient fields, which are not necessarily quasiperiodic, may satisfy~\eqref{e.ergodicassump} and~\eqref{e.technicalassump}. Recall (see~\cite{B}) that a general uniformly almost periodic function $f:\Rd \to \R$ may be written in the form
\begin{equation*} \label{}
f(x) = F(M(x))
\end{equation*}
for a continuous function $F:\R^\infty \to \R$ which is $1$--periodic in each entry and~$M:\Rd \to \R^\infty$ is a linear map. We identify $\R^\infty$ with the set of sequences of real numbers indexed by $\N$. For each $m\in\N$, we let $P_m:\R^\infty \to \R^m$ denote the projection onto the first $m$ terms of the sequence. We measure the dependence of $F$ in terms beginning with the $(m+1)$th term by setting
\begin{equation*} \label{}
\chi_m(F):= \sup_{x\in\R^\infty} \left| F(x) - F(P_mx) \right|. 
\end{equation*}
Note that $f$ is quasiperiodic if and only if $\chi_m(F) = 0$ for some finite $m$, so we can think of $\chi_m(F)$ as a measure of how well $f$ can be approximated by quasiperiodic functions. 

\smallskip

We now suppose that $F\in C^\infty(\R^\infty)$ with all derivatives bounded uniformly. We suppose that, for every $m\in\N$, there exists $\theta_m>0$ and $A_m>0$ such that $P_mM\in \mathcal D(A_m,\theta_m)$. It follows then from the above analysis that
\begin{equation} \label{e.rhokap}
\rho_k(f,R) \leq \inf_{m\in\N} \left( \chi_m(F) +  C A_m^{-\frac km} R^{\frac k{m(\theta_m+1)}} \left( \log R \right)^{\frac {k(m-1)}{m}} \right). 
\end{equation}
This leads to the estimate
\begin{align*}
\rho_*(f,R) \leq \inf_{k\in\N}\left( C^k k! \inf_{m\in\N} \left( \chi_m(F) +  C A_m^{-\frac km} R^{\frac k{m(\theta_m+1)}} \left( \log R \right)^{\frac {k(m-1)}{m}} \right) \right).
\end{align*}
Therefore a sufficient condition for~\eqref{e.ergodicassump} is
\begin{equation}
\label{e.almostperiodiccondition}
\inf_{k\in\N}\left( C^k k! \inf_{m\in\N} \left( \chi_m(F) +  C A_m^{-\frac km} R^{\frac k{m(\theta_m+1)}} \left( \log R \right)^{\frac {k(m-1)}{m}} \right) \right) \leq CR^{-1-\delta}.
\end{equation}
This may not seem very explicit, but one sees a very natural interplay between finite dimensional (quasiperiodic) approximation and the nonresonance conditions. One can give more explicit sufficient conditions for~\eqref{e.ergodicassump} by showing that~\eqref{e.almostperiodiccondition} is implied by decay conditions on the  coefficients in the Fourier series for $f$. (Recall that an arbitrary almost periodic function admits a Fourier series, see~\cite{B}.) We leave this as an exercise. One obtains an analogous condition for~\eqref{e.technicalassump} from~\eqref{e.rhokap} with $k=1$.
%
%
Therefore, it seems to be natural to use $\rho_*(\a,\cdot)$. In particular, the arguments in this paper actually give the first algebraic rates of convergence in homogenization for almost periodic coefficients which do not belong to Kozlov's quasiperiodic class.

\section*{Acknowledgements}
The second author acknowledges financial support from the European Research Council under
the European Community's Seventh Framework Programme (FP7/2014-2019 Grant Agreement
QUANTHOM 335410). The third author was supported by the Academy of Finland project \#258000.

\bibliographystyle{plain}
\bibliography{almostperiodic}

\end{document}